\documentclass[10pt,a4paper]{article}
\usepackage[latin,english]{babel}
\usepackage[utf8]{inputenc}
\usepackage[T1]{fontenc}
\usepackage{eufrak}
\usepackage{geometry}
\usepackage{amsmath}
\usepackage{bbm}
\usepackage{amssymb}
\usepackage{amsbsy}
\usepackage{amsthm}
\usepackage{makeidx}
\usepackage{mathtools}
\mathtoolsset{showonlyrefs}
\usepackage{mathabx, mathrsfs, dsfont}
\usepackage{cases}
\usepackage{braket}
\usepackage[toc,page]{appendix}
\usepackage{pdfsync}
\usepackage{hyperref}
\usepackage{graphicx}
\graphicspath{{image/}}
\usepackage{psfrag}
\usepackage{epstopdf}
\usepackage{todonotes}
\usepackage{graphicx}
\usepackage{fancyhdr}
\usepackage{math}
\usepackage{cite}
\usepackage{color}
\usepackage{subcaption}

\renewcommand{\meanval}[1]{\bE\left[#1\right]}

\newtheorem{lemma}{Lemma}[section]
\newtheorem{theorem}[lemma]{Theorem}

\newtheorem{corollary}[lemma]{Corollary}
\newtheorem{remark}[lemma]{Remark}

\numberwithin{equation}{section}

\author{
	G. Mazzuca
	\footnote{
		International School for Advanced Studies (SISSA), Via Bonomea 265,  34136 Trieste, Italy \newline
		\textit{Email: } \texttt{guido.mazzuca@sissa.it} 
}}

\title{On the mean Density of States of some matrices related to the beta ensembles and an application to the Toda lattice}

\date{\today}

\begin{document}
	\maketitle
	\begin{abstract}
		In this manuscript we  study   tridiagonal random matrix models  related  to   the classical $\beta$-ensembles (Gaussian, Laguerre, Jacobi)  in the high temperature regime, i.e. when the   size $N$ of the matrix tends to infinity
		with the constraint that $\beta N=2\alpha$   constant, $\alpha > 0$. 
		We call these ensembles the Gaussian, Laguerre and Jacobi $\alpha$-ensembles and we prove the convergence of their empirical spectral distributions to their mean densities of states and we compute them explicitly. As an application we explicitly compute the mean density of states of the  Lax matrix of the  Toda lattice with periodic boundary conditions with respect to the Gibbs ensemble.
	\end{abstract}

	\section{Introduction}
	In this manuscript  we consider  some tridiagonal random matrix  models  related to the classical $\beta$-ensembles \cite{Dumitriu2002,Killip2004,Edelman2008}. More specifically we  study  the mean density of states of the  random  matrices in Table \ref{tab:alpha_ensemble} where the quantity $\cN(0,\sigma^2)$ is  the real Gaussian  random variable with density $\frac{e^{-\frac{x^2}{2\sigma^2}}}{\sqrt{2\pi\sigma^2}}$ {supported on all $\R$}, the  quantity  
	$\chi_{2\alpha}$ is the chi-distribution  with density $\frac{x^{2\alpha-1}e^{-\frac{x^2}{2}}}{2^{\alpha-1}\Gamma(\alpha)}$ {supported on $\R^+$, here $\Gamma(\alpha)$ is  the gamma function,} and $\text{Beta}(a,b)$ is  the Beta random variable  with density $\frac{\Gamma(a+b)x^{a-1}(1-x)^{b-1}}{\Gamma(a)\Gamma(b)}$ {supported on $(0,1)$}.

	\begin{table}[h]
		
		\centering
		\begin{tabular}{|l|l|}
			\hline
			Gaussian $\alpha$ Ensemble   & $H_\alpha \sim \frac{1}{\sqrt{2}} \begin{pmatrix}
			a_1 & b_1 \\
			b_1 & a_2& b_2 \\
			& \ddots & \ddots & \ddots\\
			&& b_{N-1} & a_N
			\end{pmatrix}\, ,$\\
			& $H_\alpha \in \text{Mat}(N\times N)\,,$\\
			& $b_n \sim \chi_{2\alpha} \quad n=1,\ldots, N-1\,,$\\
			& $a_n \sim \cN(0,2)\quad n=1,\ldots,N\,,$
			\\ \hline
			Laguerre $\alpha$ Ensemble & $L_{\alpha,\gamma} = B_{\alpha,\gamma} B_{\alpha,\gamma}^\intercal , \quad B_{\alpha,\gamma} =\frac{1}{\sqrt{2}} \begin{pmatrix}
			x_1 \\
			y_1 & x_2 \\
			& \ddots & \ddots  \\
			&& y_{N-1} & x_N
			\end{pmatrix},$ \\
			& $B_{\alpha,\gamma} \in \text{Mat}(N\times M),\, M \geq N,$\\
			& $x_n \sim \chi_{\frac{2\alpha}{\gamma}}\quad n = 1, \ldots, N,$\\
			& $y_n \sim \chi_{2\alpha}\quad n= 1, \ldots, N-1,$ \\ \hline
			Jacobi $\alpha$ Ensemble   & $J_\alpha = D_\alpha D_\alpha^\intercal , \quad  D_\alpha = \begin{pmatrix}
			s_1 \\
			t_1 & s_2 \\
			& \ddots & \ddots \\
			&& t_{N-1} & s_N
			\end{pmatrix},$ \\
			& $ D_\alpha \in \text{Mat}(N\times N),$\\
			&  $t_n = \sqrt{q_n(1-p_n)},\, s_n = \sqrt{p_n(1-q_{n-1})},$\\
			&  $q_n \sim  \text{Beta}\left(\alpha, \alpha + a+b +2\right) \quad (q_0 = 0),$ \\ 
			& $p_n \sim  \text{Beta}\left(\alpha + a + 1, \alpha +b +1\right).$
			\\ \hline
		\end{tabular}
		\caption{The Gaussian, Laguerre and Jacobi $\alpha$-ensembles.}
		\label{tab:alpha_ensemble}
	\end{table}
	
	Let us explain some terminology first and then state our result.
	
	 A random Jacobi matrix is  a symmetric tridiagonal  $N\times N$ matrix of the form
	\begin{equation}
	\label{eq:jacobi}
	T_N := \begin{pmatrix}
	a_1 & b_1 \\
	b_1 & a_2 & b_2\\
	& \ddots & \ddots & \ddots \\
	&& \ddots & \ddots & b_{N-1}\\
	&&&b_{N-1} & a_N
	\end{pmatrix}\, 
	\end{equation} 
	where $\{a_i\}_{i=1}^N$  are i.i.d. real random variables and $\{b_i\}_{i=1}^{N-1}$ are  i.i.d. positive random variables independent from the $a_i$.
	This matrix has the property of having $N$-distinct eigenvalues \cite{Deift}.
	The empirical spectral distribution of $T_N$  is the random probability distribution on $\mathbb{R}$ defined as 
	\begin{equation}
	\label{empirical_m}
	\di \nu^{(N)}_{T} := \frac{1}{N}\sum_{j=1}^N\delta_{\lambda^{(N)}_j}\,,
	\end{equation}
	where { $\lambda_1^{(N)}>\,\dots\,>\lambda_N^{(N)}$} are the eigenvalues of $T_N$ and $\delta_{(\cdot)}$ is the delta function.  
	{The mean {\em Density of States }  $\wo{\di \nu_{T}}$  is the non random probability distribution, provided it exists, defined as }
	\begin{equation}
	\label{eq:def_bar}
	\la \wo{\di \nu_{T}}, f \ra  := \lim_{N\to \infty}\meanval{\la {\di \nu^{(N)}_{T}}, f \ra}\, ,
	\end{equation} 
	{for all continuous and bounded functions $f$, here  $\la \di\sigma, f \ra := \int_\R f \di \sigma $  and $ \meanval{\cdot}$ stands for the expectation with respect to the given probability distribution.}
	
	In this manuscript   we  identify the mean density of states of the  Gaussian, Laguerre and Jacobi $\alpha$-ensembles  introduced in Table~\ref{tab:alpha_ensemble}.
	
	\begin{theorem}
		\label{THM:MAIN}
		Consider the matrices $ H_\alpha, L_{\alpha,\gamma}$, and  $J_\alpha$  in Table~\ref{tab:alpha_ensemble} with $\alpha\geq 0$, $ \gamma\in(0,1)$, $a+\alpha > 0$, $b+\alpha > 0$ and $a\not\in \N$. Then {their empirical spectral distributions $\di\nu_H^{(N)},\di\nu_L^{(N)},$ and $\di\nu_J^{(N)}$
		 converge almost surely,  in the large $N$ limit,  to their corresponding mean density of states, whose formula are given explicitly by}:
		
		\begin{align}
		\label{eq:mean_gaussian}
		&\wo{\di \nu_H}(x) = \partial_{\alpha}(\alpha \mu_{\alpha}(x)) \di x \, ,\\
		\label{eq:mean_laguerre}
		&\wo{\di \nu_L}(x) = \partial_\alpha\left(\alpha  \mu_{\alpha,\gamma}(x)\right) \di x\, ,\quad x \geq 0 ,\\
		\label{eq:mean_jacobi}
		&\wo{\di \nu_J} (x)= \partial_{\alpha}\left(\alpha  \mu_{\alpha,a,b}(x)\right) \di x\, , \quad 0 \leq x \leq 1 \, .
		\end{align}	
		Here $\partial_{\alpha}$ is the derivative with respect to $\alpha$ and 
		\begin{align}
		\label{eq:spectralGaussian}
		&\mu_\alpha(x):=  \frac{e^{-\frac{x^2}{2}}}{\sqrt{2\pi}} \abs{\wh f_{\alpha}(x)}^{-2}\, , \quad 	 
		\wh f_{\alpha}(x): =\sqrt{\frac{\alpha}{\Gamma(\alpha)}}\int_{0}^{\infty} t^{\alpha - 1}e^{-\frac{t^2}{2}} e^{i x t}\di t
		\, ,\\
		\label{eq:spectralLaguerre}
		& \mu_{\alpha,\gamma}(x) := \frac{1}{\Gamma(\alpha+1)\Gamma\left(1+ \frac{\alpha}{\gamma} + \alpha\right)} \frac{x^\frac{\alpha}{\gamma}e^{-x}}{\Big\lvert \psi\left(\alpha,-\frac{\alpha}{\gamma};xe^{-i\pi}\right)\Big\rvert^2} \quad x \geq 0\, ,
		\end{align}
		with  $\Gamma(z)$ the gamma-function and  $\psi(v,w;z)$ is  the Tricomi's confluent hypergeometric function, for the definition see Appendix \ref{appendixA},
		and 	
		\begin{equation}
		\label{eq:explicit_jacobi_TT}
		\mu_{\alpha,a,b}(x) := 	\frac{\Gamma(\alpha +1)\Gamma(\alpha + a + b +2 )}{\Gamma(\alpha + a+1)\Gamma(\alpha + b +1)}\frac{x^a(1-x)^b}{\Big \lvert U(x)+ e^{i \pi b }V(x)\Big \rvert^2}\,\qquad 0 \leq x \leq 1 ,
		\end{equation}
		where
		
		\begin{align}
		&U(x) := \frac{\Gamma(\alpha + 1)\Gamma(a+1)}{\Gamma(1+\alpha+a)}{_2F_1}(\alpha,-\alpha-a-b-1,-a;x) \,,\\
		&V(x) := \frac{-\pi \alpha\Gamma(\alpha + a+ b +2)}{\sin(\pi a)\Gamma(1+\alpha +b)\Gamma(a+2)}(1-x)^{b+1}x^{a+1} {_2F_1}(1-\alpha,\alpha+a+b+2,2+a;x)\, ,
		\end{align}
		here ${_2F_1}(a,b,c;z)$ is the Hypergeometric function:
		
		\begin{equation}
		{_2F_1} (a,b,c;z) := \sum_{n=0}^{\infty} \frac{(a)_n(b)_n}{(c)_n}\frac{z^n}{n!}\, , \quad  (a)_n:=a(a+1)\cdots(a+n-1)\, .
		\end{equation}
		Moreover, for any non trivial polynomial $P(x)$ the following limits hold:
		\begin{align}	
		\label{eq:fluct_gaussian}	
		&\sqrt{N}\left(\la \di \nu_H^{(N)}, P(x) \ra - \la \wo{\di \nu_H},P(x)\ra\right) \stackrel{d}{\to} \cN(0,\sigma_P^2) \, \quad \text{as } \, N\to \infty ,\\
		\label{eq:fluct_laguerre}
		&\sqrt{N}\left(\la \di \nu_L^{(N)}, P(x) \ra - \la \wo{\di \nu_L},P(x)\ra\right) \stackrel{d}{\to} \cN(0,\wo \sigma_P^2) \,\quad \text{as } \, N\to \infty ,\\
		\label{eq:fluct_jacobi}
		&\sqrt{N}\left(\la \di \nu_J^{(N)}, P(x) \ra - \la \wo{\di \nu_J},P(x)\ra\right) \stackrel{d}{\to} \cN(0,\wt \sigma_P^2)\, \quad \text{as } \, N\to \infty ,
		\end{align}
				for some constants $\sigma^2_P,\wo \sigma^2_P,\wt \sigma^2_P \geq 0$, here $\stackrel{d}{\to}$ is the convergence in distribution.

	\end{theorem}
	In figures \ref{fig:gaussian}--\ref{fig:jacobi} we plot the empirical spectral distribution of the $\alpha$-ensembles for different values of the parameters.

		The measures {with density }$\mu_\alpha$,   $\mu_{\alpha,\gamma}$ and $\mu_{\alpha,a,b}$ have already appeared in the literature as the orthogonality measures of the associated Hermite, Laguerre and Jacobi polynomias ( see Appendix~\ref{appendixA}). Such measures  have  also  appeared  
		in the study of  the  classical $\beta$-ensembles  \cite{Dumitriu2002} (see Table \ref{tab:beta_ensemble})  in 
		the {\em high temperature regime}, namely in the limit when  $N\to \infty$, with $\beta N \to 2\alpha$, $\alpha >0$, \cite{BP, Duy2015, Trinh2019,Trinh2020,Allez2012,Allez2013,Hardy2020}.  
		In order to summarize the results of those papers we recall that 
		for the Jacobi matrix $T_N$ in \eqref{eq:jacobi} the {\it  spectral measure}   $\di \mu^{(N)}_{T}$ is   the probability measure supported on its eigenvalues 
		$\lambda^{(N)}_1,\ldots,\lambda^{(N)}_N$ with weights $q_1^2,\ldots,q_N^2$ where  $q_j= |\langle v_j^{(N)},e_1\rangle|$ and $v^{(N)}_1,\ldots, v^{(N)}_N$  are the orthonormal 
		eigenvectors:
		\begin{equation}
		\label{spectral_m}
		\di \mu^{(N)}_{T} : =\sum_{j=1}^Nq^2_j\delta_{\lambda^{(N)}_j}.
		\end{equation}
		As the eigenvectors form an orthonormal basis, and  $||e_1||=1$ we get that $\sum_{j=1}^Nq_j^2=1$.
		Moreover the set of finite Jacobi matrix of size $N$ is in one to one correspondence with the set of probability measure supported on $N$ real points \cite{Deift}.
		
		For the $\beta$-ensembles the quantities  $\{q_i\}_{i=1}^N$ are independent from the eigenvalues and are distributed as $(\chi_\beta, \ldots, \chi_\beta)$ normalized to unit length \cite{Dumitriu2002,Killip2004,Edelman2008}. It follows that $E[q_j^2]=\frac{1}{N}$. Consequently the mean of the empirical measure \eqref{empirical_m} coincides with the mean of the spectral measure \eqref{spectral_m}, namely 
		$$\di \bar{\nu}^{(N)}_{H_\beta}  =\di\bar{ \mu}^{(N)}_{H_\beta},\;\;\di \bar{\nu}^{(N)}_{L_\beta}  =\di\bar{ \mu}^{(N)}_{L_\beta},\;\;\;\di \bar{\nu}^{(N)}_{J_\beta}  =\di\bar{ \mu}^{(N)}_{J_\beta},$$
		where $H_\beta$, $L_\beta$ and $J_\beta$ refer to the Hermite, Laguerre and Jacobi $\beta$-ensembles.
		It is shown in \cite{Duy2015}  ( see also \cite{BP,Allez2012})  that the measures $\di \bar{\nu}^{(N)}_{H_\beta}=\di\bar{ \mu}^{(N)}_{H_\beta}$  converge  weakly, in the limit   $N\to \infty$, with $\beta N= 2\alpha$,  to the non random probability measure  with density $\mu_\alpha$ defined in \eqref{eq:spectralGaussian}.
		It is shown in \cite{Trinh2019,Trinh2020,Allez2013}   that the measures $\di \bar{\nu}^{(N)}_{L_\beta}  =\di\bar{ \mu}^{(N)}_{L_\beta}$ and $\di \bar{\nu}^{(N)}_{J_\beta}  =\di\bar{ \mu}^{(N)}_{J_\beta}$, under some mild assumptions on the parameters,   converge weakly in the limit  $N\to \infty$, with $\beta N \to 2\alpha$ and $ N/M \to \gamma\in (0,1)$ to the non random probability measures with  density  $  \mu_{\alpha,\gamma}$   and $ \mu_{\alpha,a,b}$  defined in  \eqref{eq:spectralLaguerre}   and \eqref{eq:explicit_jacobi_TT}   respectively.
		In \cite{Duy2015,Trinh2019,Trinh2020} it is showed that these measures coincide with the mean spectral measures  of the random matrices $H_\alpha$, $L_\alpha$ and $J_\alpha$, see Table \ref{tab:alpha_ensemble}.
	
The problem of  convergence of the empirical   spectral distribution of  the Gaussian, Laguerre and Jacobi $\alpha$-ensembles has remained unsolved.
	The present manuscript addresses such problem in Theorem~\ref{THM:MAIN} by determining the mean Density of States of such random matrices   and their fluctuation.
	Our strategy to prove the result is the application of the moment method and an astute counting  of the super-Motzkin paths \cite{Joris2015} to calculate the moments of the  the Gaussian, Laguerre and Jacobi $\alpha$ and $\beta$-ensembles.

	For completeness we  mention  also the result  in \cite{Popescu2009}  where a different generalization of the Gaussian $\beta$ ensemble is studied. Indeed   in \cite{Popescu2009} the author examined the mean spectral measure of a random Jacobi matrix $T_N$ such that there exists a sequence of real number $\{m_k\}_{k\geq 0}$ and $m_0 = 1$ such that $\meanval{(b_1/N^\sigma)^k} \to m_k$ as $N\to \infty$ for all fixed  $k\in \N$, which is a generalization of the classical case where $b_1 \sim \chi_{\beta(N-1)}/\sqrt{2}$ and $\sigma = 1/2$.
	\begin{table}[ht]
		\centering
		\begin{tabular}{|l|l|}
			\hline
			Gaussian $\beta$ Enseble  & $H_\beta \sim \frac{1}{\sqrt{2}} \begin{pmatrix}
			a_1 & b_1 \\
			b_1 & a_2& b_2 \\
			& \ddots & \ddots & \ddots\\
			&& \ddots & \ddots & b_{N-1}\\
			&&& b_{N-1} & a_N
			\end{pmatrix}$\\
			& $H_\beta \in \text{Mat}(N\times N),$\\
			& $b_n \sim \chi_{\beta(N-n)} \quad n=1,\ldots, N-1,$\\
			& $a_n \sim \cN(0,2)\quad n=1,\ldots,N,$
			\\ \hline
			Laguerre $\beta$ Enseble   & $L_{\beta,\gamma} = B_{\beta,\gamma} B_{\beta,\gamma}^\intercal , \quad B_{\beta,\gamma} =\frac{1}{\sqrt{2}} \begin{pmatrix}
			x_1 \\
			y_1 & x_2 \\
			& \ddots & \ddots  \\
			&& y_{N-1} & x_N
			\end{pmatrix},$ \\
			& $B_{\beta,\gamma} \in \text{Mat}(N\times M),\, M \geq N,$\\
			& $x_n \sim \chi_{\beta(M-n+1)}\quad n = 1, \ldots, N,$\\
			& $y_n \sim \chi_{\beta (N-n)}\quad n = 1, \ldots, N-1,$ \\ \hline
			Jacobi $\beta$ Enseble   & $J_\beta = D_\beta D_\beta^\intercal , \quad  D_\beta = \begin{pmatrix}
			s_1 \\
			t_1 & s_2 \\
			& \ddots & \ddots \\
			&& t_{N-1} & s_N
			\end{pmatrix},$ \\
			& $ D_\beta \in \text{Mat}(N\times N),$\\
			&  $t_n = \sqrt{q_n(1-p_n)},\, s_n = \sqrt{p_n(1-q_{n-1})},$\\
			&  $q_n \sim  \text{Beta}\left(\frac{\beta(N-n)}{2}, \frac{\beta(N-n)}{2} + a+b +2\right) \quad (q_0 = 0),$ \\ 
			& $p_n \sim  \text{Beta}\left(\frac{\beta(N-n)}{2} + a + 1, \frac{\beta(N-n)}{2} +b +1\right).$
			\\ \hline
		\end{tabular}
		\caption{The Gaussian, Laguerre and Jacobi $\beta$-ensembles}
		\label{tab:beta_ensemble}
	\end{table}

	Finally  we  relate the Gibbs ensemble of the classical Toda chain to  the Gaussian $\alpha$-ensemble. In particular we  obtain,  as a corollary of Theorem \ref{THM:MAIN}, the mean density of states of the Toda  Lax matrix  with periodic boundary conditions  
	when the matrix entries are distributed accordingly to the Gibbs ensemble { and when the number of particles goes to infinity}.
	This result is instrumental to study the Toda lattice in the thermonodynamic limit.
	We remark that the mean density of states of the Toda Lax matrix  has already  appear in the physics literature  \cite{Spohn2019}. Here we present an alternative proof of this result.

	\section{Preliminary results}
	
	In this section we summarize some known results and techniques that we will use along the proof { of the main theorem.}
	
	The moments of a measure $\di \sigma$, when they exist, are defined as:
	
	\begin{equation}
	\label{eq:def_mom}
	u^{(l)} := \la x^l, \di \sigma \ra \quad l\in\N\, .
	\end{equation} 
	Under some mild  assumptions, they totally define the measure itself, indeed the following Lemma, whose proof can be found in \cite[Lemma B.2]{Bai2010}, holds:
	\begin{lemma}(cf . \cite[Lemma B.2]{Bai2010})
		\label{lem:bay}
		Let $\{u^{(l)}\}_{l\geq 0}$ be the sequence of moments of a measure $\di \sigma$. If 
		\begin{equation}
		\label{eq:divergence_moments}
		\lim_{l\to \infty}\inf\frac{(u^{(2l)})^\frac{1}{2l}}{l} < \infty\, ,
		\end{equation}
		then $\di \sigma$ is uniquely determined by the moment sequence $\{u^{(l)}\}_{l\geq 0}$.
	\end{lemma}
	
	This implies that if two measures have the same moment sequence and \eqref{eq:divergence_moments} holds then the two measure are the same. We will exploit this  property,  indeed we will show that the moments of the random matrices $H_\alpha,L_\alpha$ and $J_\alpha$ coincide, in the large $N$ limit, with the moments of  the measure $\wo{\di \nu_H}(x), \wo{\di \nu_L}(x)$ and $\wo{\di \nu_J}(x)$ in \eqref{eq:mean_gaussian}--\eqref{eq:mean_jacobi} and we will prove that \eqref{eq:divergence_moments} holds for all of them. This technique undergoes the name of moment method. 
	
	In order to apply this idea, we need to compute explicitly the moments of the mean density of states for the $\alpha$ and $\beta$-ensembles.  We will use the following identity for the moments of the mean density of states:
	
	\begin{equation}
	\label{eq:alternative_moment}
	\la \wo{\di \nu_{T}}, x^l \ra  = \lim_{N\to \infty}\frac{1}{N}\meanval{\mathit{Tr}(T^l_N)}\ \, ,
	\end{equation} 
	where
	\[
	\mathit{Tr}(T^l_N):=\sum_{j=1}^NT^l_N(j,j),
	\]
	and $T_N^l(j,i)$ is the entry $(j,i)$ of the matrix $T_N^l$ and the average is made according to the distribution of the matrix entries. From now on we will write $\meanval{f(\ba,\bb)}_T$ as the mean value of $f(\ba,\bb)$ made according to the distributions of the matrix $T$'s  {entries}, here $\ba$ is a vector of components $a_1,\, \ldots, \, a_N$.
	
	To conclude the computation of the moments, we need an explicit expression for the terms $T_N^l(j,j)$.  The following lemma proved in \cite{Grava2020} gives us their general expressions:
	
	\begin{theorem}(cf. \cite[Theorem 3.1]{Grava2020})
		\label{thm:structure_moments}
		For any {$1 \leq l  < N$}, consider the tridiagonal matrix $T_N$ \eqref{eq:jacobi},	{then} one has 
		\begin{equation}
		\label{eq:Jacobi_trace_periodic}
		\mathit{Tr}(T_N^l)= \sum_{j=1}^N h_{j}^{(l)} \, ,
		\end{equation}		
		where $ h_{j}^{(l)}:= T^l_N(j,j)$ is given explicitly for $ \lfloor{l/2}\rfloor < j <N- \lfloor{l/2}\rfloor$ by
		\begin{equation}\label{eq:general_super_motzkin}
		h_{j}^{(l)} (\bb,\ba)= \sum_{(\bn,\bk)\in \cA^{(l)}} \, \rho^{(l)}(\bn,\bk) \prod_{i = - \lfloor{l/2}\rfloor}^{ \lfloor{l/2}\rfloor-1} b_{j+i}^{2{n_i}} \prod_{i = - \lfloor{l/2}\rfloor+1 }^{ \lfloor{l/2}\rfloor -1} a_{j+i}^{k_i} \, .
		\end{equation}
		Here $\cA^{(m)}$ is the set  
		
		\begin{equation}
		\label{cAm}
		\begin{split}
		\cA^{(l)} := \Big\{(\bn,\bk) \in \N^{\mathbb{Z}}_0 \times \N^{\mathbb{Z}}_0 \ \colon \ \ \ 
		& \sum_{i= - \lfloor{l/2}\rfloor }^{ \lfloor{l/2}\rfloor-1} \left(2n_i + k_i\right) = l  , \\
		& \forall i \geq 0, \ \ \ n_i = 0 \Rightarrow n_{i+1} = k_{i+1} = 0,  \,  
		\\
		& \forall i < 0, \ \ \ n_{i+1} = 0 \Rightarrow n_{i}= k_i = 0  
		\Big\}.
		\end{split}
		\end{equation}
		The quantity  $\N_0=\N\cup\{0\}$
		and $\rho^{(l)}(\bn, \bk) \in \N $ is 
		given by 
		\begin{align}
		\label{rhom}
		\rho^{(l)}(\bn,\bk) := &\binom{n_{-1} + n_0 + k_0}{k_0}\binom{n_{-1} + n_0}{n_0}
		\prod_{i=- \lfloor{l/2}\rfloor \atop i \neq -1}^{  \lfloor{l/2}\rfloor -1}\binom{n_i + n_{i+1} +k_{i+1} -1}{k_{i+1}}\binom{n_i + n_{i+1} -1}{n_{i+1}} \, .
		\end{align}		
	\end{theorem}
	
	\begin{remark}
		\label{rem:tutto_finito}
		Formula \eqref{eq:general_super_motzkin} holds for $ \lfloor{l/2}\rfloor < j < N -  \lfloor{l/2}\rfloor$, for the other values of $j$ the formula is slightly different.  This  is because for $j\leq \lfloor l/2\rfloor$ or $j \geq N - \lfloor l/2\rfloor$ the polynomial $h_j^{(l)}$ is related to a constrained Super Motzkin path,  \cite{Mansour2013},  instead for $ \lfloor{l/2}\rfloor < j < N -  \lfloor{l/2}\rfloor$ it is related to a classical Super Motzkin path. In any case the polynomial $h_j^{(l)}$ is independent from
		$N$ for all $j$. 
	\end{remark}
	
	We remark that both $|\cA_l|$ and $\rho^{(l)}(\bn,\bk)$ do not depend on $N$ and $j$. Moreover,  from the condition $\sum_{i= - \lfloor{l/2}\rfloor }^{ \lfloor{l/2}\rfloor-1} \left(2n_i + k_i\right) = l$ in \eqref{cAm} one gets that 
	\begin{equation}
	\label{rem:hmj_a}
	\begin{aligned}
	&l \text{ even} \quad \Longrightarrow \quad h^{(l)}_j \text{ contains only even polynomials in } \ba, \\
	&l \text{ odd} \quad \Longrightarrow \quad h^{(l)}_j \text{ contains only odd polynomials in } \ba.
	\end{aligned}
	\end{equation}
	
	To prove the almost sure convergence of the empirical spectral distributions $\di\nu_H^{(N)}, \di\nu_L^{(N)}$ and $\di\nu_J^{(N)}$ to their corresponding mean density of states, we will use two general results. The first one is the following Theorem proved in \cite{Nakano2020}:
		
	\begin{theorem}(cf. \cite[Theorem 2.2]{Nakano2020})
		\label{thm:nakano}
		Consider a random Jacobi matrix $T_N$ \eqref{eq:jacobi} and  assume that  $\{a_n\}_{n=1}^N$ and $\{b_n\}_{n=1}^{N-1}$ have all finite moments. Then for any non trivial polynomial $P(x)$:
		\begin{align}
		\label{eq:almost_sure_mom}
		&\la \di \nu_T^{(N)}, P(x) \ra \stackrel{a.s.}{\to}  \la \wo{\di \nu_T},P(x)\ra \quad \text{ as } \, N\to \infty\, \\
		\label{eq:fluctuations}
		&\sqrt{N}\left(\la \di \nu_T^{(N)}, P(x) \ra - \la \wo{\di \nu_T},P(x)\ra\right) \stackrel{d}{\to} \cN(0,\sigma_P^2) \, \quad \text{as } \, N\to \infty ,
		\end{align}
		for some constant $\sigma_P^2 \geq 0$. Here $\stackrel{a.s.}{\to}$ is the almost sure convergence and $\stackrel{d}{\to}$ is the convergence in distribution.
	\end{theorem}

We observe  that  Theorem~\ref{thm:nakano} is not  stated in the present form in  \cite{Nakano2020}  but this formulation is more convenient for our analysis. The second result is the following classical Lemma whose proof can be found in \cite{Duy2018,AndersonGuillonet}:
	
\begin{lemma}(cf. \cite[Lemma 2.2]{Duy2018})
	\label{lem:duy}
	Consider a sequence of random probability measures $\{\di \mu_n\}_{n=1}^\infty$ and $\di\mu$ a probability measure determined by its moments according to Lemma \ref{lem:bay}. Assume that any moment  of $\di\mu_n$ converges almost surely to the one of $\di\mu$. Then as $n\to \infty$ the sequence of measures $\{\di\mu_n\}_{n=1}^\infty$ converges weakly, almost surely, to $\di\mu$, namely for all bounded and continuous functions  $f$:
	
	\begin{equation}
	\la\di\mu_n, f \ra \to  \la\di\mu, f \ra \quad \text{a.s as } \, n\to \infty\, .
	\end{equation} 
	The convergences still holds for a continuous function $f$ of polynomials growth.
\end{lemma}

	Finally, before moving to the actual proof of our main theorem, we summarize the main results of \cite{Duy2015,Trinh2019,Trinh2020,Allez2012,Allez2013} in the following theorem.
	\begin{theorem}
		\label{thm:ht_regime}
		As $N\to \infty, \beta N \to 2\alpha\in (0,\infty), \frac{N}{M} \to \gamma \in (0,1)$, $a+\alpha > 0$, $b+\alpha > 0$ and $a\not\in \N$, the mean spectral measure and the mean density of state of the Gaussian, Laguerre and Jacobi $\beta$-ensembles  weakly converge to  the non random measures with density $ \mu_{\alpha}(x), \mu_{\alpha,\gamma}(x)$ and $ \mu_{\alpha,a,b}(x) $   defined in \eqref{eq:spectralGaussian}, \eqref{eq:spectralLaguerre} and \eqref{eq:explicit_jacobi_TT} respectively.
		Moreover \eqref{eq:divergence_moments} holds for their moments sequences. 
	\end{theorem}
	
	%We remark  that the above  theorem is not originally stated in this form. Indeed Theorem \ref{thm:ht_regime} is a summary of three different results contained in \cite{Duy2015,Trinh2019,Trinh2020}, but this equivalent form is most suitable for our analysis.
	
	\section{Proof of the main result}
	
	We are now in position to prove our main result. First of all we remark that the density $\partial_{\alpha}(\alpha\mu_\alpha(x)),\, \partial_{\alpha}(\alpha\mu_{\alpha,\gamma}(x))$ and $\partial_{\alpha}(\alpha\mu_{\alpha,a,b}(x))$ define a 
	probability  measure since the densities $\mu_\alpha(x),\, \mu_{\alpha,\gamma}(x)$ and $\mu_{\alpha,a,b}(x)$ define a probability measure. Then, since we want to apply the moment method, we have to compute the moments of the  $\alpha$-ensembles explicitly. To conclude the proof we will need also an explicit expression of the moments of the mean density of states of the $\beta$-ensembles. The following lemma lays the ground to conclude both computations.
	\begin{lemma}
		\label{lem:relation}
		Fix $\alpha \in \R^+\setminus\{0\},\, \gamma \in(0,1),a,b > -1$ and $N/2 > l \in \N$. Consider the $\alpha$ and $\beta$-ensembles in Table \ref{tab:alpha_ensemble}-\ref{tab:beta_ensemble},
		there exist polynomials $w_l(x), g_l(x)$, and rational and continuous functions $r_l(x)$ such that, for $N$ large enough and $\beta N = 2\alpha, \frac{N}{M} = \gamma$,  the following holds, for $\lfloor{l/2}\rfloor<j<N-\lfloor{l/2}\rfloor$ :
		
		\begin{align}
		\label{eq:mom_gaussian_beta}
		&\meanval{h_j^{(l)} }_{H_\beta } = \begin{cases}
		w_l\left(\alpha\left(1-\frac{j}{N}\right)\right)  + O\left(N^{-1}\right) \qquad &l \text{ even}\\
		0 \qquad &l \text{ odd}
		\end{cases} ,\\
		\label{eq:mom_gaussian_alpha}
		&\meanval{h_j^{(l)} }_{H_\alpha} = \begin{cases}
		w_l(\alpha )  \qquad& l \text{ even}\\
		0 \qquad &l \text{ odd}
		\end{cases}\,,\\
		&\meanval{h_j^{(l)} }_{L_\beta } = g_l\left(\alpha\left(1-\frac{j}{N}\right)\right)  + O\left(N^{-1}\right)\, ,\\
		&\meanval{h_j^{(l)} }_{L_\alpha}= g_l(\alpha )	\, , \\
		&\meanval{h_j^{(l)} }_{J_\beta} = r_l\left(\alpha\left(1-\frac{j}{N}\right)\right)  + O\left(N^{-1}\right)\, ,\\
		&\meanval{h_j^{(l)} }_{J_\alpha} = r_l(\alpha )\, .	
		\end{align}
	\end{lemma}

	\begin{proof}[Proof of Lemma \ref{lem:relation}]
		
		We will just prove \eqref{eq:mom_gaussian_beta}-\eqref{eq:mom_gaussian_alpha} since the proof of the other cases is similar. Indeed the only difference in the proofs is that for the Gaussian and Laguerre $\alpha$ and $\beta$-ensembles we use the fact that the expected value of any even monomial with respect to a $\chi_\xi$-distribution is a monomial in $\xi$. While for the Jacobi $\alpha$ and $\beta$-ensembles we use the fact that the expected values of any monomial with respect to a $\text{Beta}(a,b)$-distribution is a rational functions of the parameters.
		
		First of all, since {$\ba=(a_1,\dots,a_N)$} are normal distributed for both ensembles and thanks to \eqref{rem:hmj_a} we get that
		
		\[ \meanval{h_j^{(l)} }_{H_\alpha } = \meanval{h_j^{(l)} }_{H_\beta} = 0\,, \qquad l \text{ odd}\, .\]
		
		For the Gaussian $\alpha$ ensemble we have that, for $ \lfloor{l/2}\rfloor < j < N- \lfloor{l/2}\rfloor$
		$$\meanval{h_j^{(l)}}_{H_\alpha} = \meanval{\sum_{(\bn,\bk)\in \cA^{(l)}} \, \rho^{(l)}(\bn,\bk) \prod_{i = - \lfloor{l/2}\rfloor}^{ \lfloor{l/2}\rfloor-1} b_{j+i}^{2{n_i}} \prod_{i = - \lfloor{l/2}\rfloor+1 }^{ \lfloor{l/2}\rfloor -1} a_{j+i}^{k_i}}_{H_\alpha}$$
		is independent from $j$ since $b_{j+i}\sim \chi_{2\alpha}, a_i \sim\cN(0,2)\,, \, i = -\lfloor l/2 \rfloor, \ldots, \lfloor l/2 \rfloor$, and the coefficients $\rho^{(l)}(\bn,\bk)$ and the set $\cA^{(l)}$ are independent from $j$ and $N$ by Theorem \ref{thm:structure_moments}. {Moreover, as already pointed out,  the expected values of any even monomial with respect to a $\chi_\xi$-distribution is a monomial in $\xi$.} Thus we have that for fixed $l\in\N$, there exists a polynomial $w_l(\alpha)$ such that \eqref{eq:mom_gaussian_alpha} holds.
		
		We can apply a similar reasoning for the Gaussian $\beta$ ensemble, indeed we notice that if we approximate the distribution of $b_{j+i}\sim\chi_{2\alpha\left(1-\frac{j+i }{N}\right)}$, $i=- \lfloor{l/2}\rfloor, \ldots,  \lfloor{l/2}\rfloor$ with the one of $b_{j}\sim \chi_{2\alpha\left(1-\frac{j}{N}\right)}$ we get an error of order $N^{-1}$ when we evaluate the expected value. So we can compute 
		\begin{equation}
		\begin{split}
		\meanval{h_j^{(l)} }_{H_\beta } &= \meanval{\sum_{(\bn,\bk)\in \cA^{(l)}}\, \rho^{(l)}(\bn,\bk) \prod_{i = - \lfloor{l/2}\rfloor}^{ \lfloor{l/2}\rfloor-1} b_{j+i}^{2{n_i}} \prod_{i = - \lfloor{l/2}\rfloor+1 }^{ \lfloor{l/2}\rfloor -1} a_{j+i}^{k_i}}_{H_\beta}\\ &= w_l\left(\alpha\left(1-\frac{j}{N}\right)\right)  + O\left(N^{-1}\right) \, ,
		\end{split}
		\end{equation}
		
		where  the only difference from the previous case is that the parameter of the $\chi$-distribution is $2\alpha\left(1-\frac{j}{N}\right)$ instead of $2\alpha$.
		
	\end{proof}  
	
Using the above  lemma we can conclude the computation of the moments for the $\alpha$ and $\beta$-ensembles: 
	\begin{corollary}
		\label{cor:momenti_alpha}
		
		Fix $l \in \N$, $\alpha \in \R^+\setminus\{0\}$, $a,b > -1$ and $\gamma \in (0,1)$ then in the large $N$ limit, with $N\beta \to 2\alpha$ and $ \frac{N}{M}\to \gamma$, the following holds:
		
		\begin{align}
		\label{eq:final_mom_alpha}
		& u^{(l)}_{\alpha} := \lim_{N\to \infty} \meanval{ \frac{1}{N}\mathit{Tr}(H_\beta^{l}) }_{H_\beta}= \begin{cases}
		\int_{0}^1w_{l}\left(\alpha x\right)  \di x\qquad l \text{ even}\\
		0\qquad l \text{ odd}
		\end{cases}\, , \\ 
		\label{eq:final_mom_beta}
		&v^{(l)}_{\alpha} := \lim_{N\to \infty} \meanval{ \frac{1}{N}\mathit{Tr}(H_\alpha^{l}) }_{H_\alpha } =\begin{cases}
		w_{l}\left(\alpha \right) \qquad l \text{ even}\\
		0 \qquad l \text{ odd}
		\end{cases}\, ,\\
		&u^{(l)}_{\alpha,\gamma} := \lim_{N\to \infty} \meanval{ \frac{1}{N}\mathit{Tr}(L_\beta^l) }_{L_\beta }=\int_{0}^{1}g_l(\alpha x  ) \di x\, , \\ 
		& v^{(l)}_{\alpha,\gamma} := \lim_{N\to \infty} \meanval{ \frac{1}{N}\mathit{Tr}(L_\alpha^l) }_{L_\alpha} =g_l(\alpha )\, , \\
		& u^{(l)}_{\alpha,a,b} := \lim_{N\to \infty} \meanval{ \frac{1}{N}\mathit{Tr}(J_\beta^l) }_{J_\beta}=  \int_{0}^{1}r_l(\alpha x ) \di x \, ,\\ 
		& v^{(l)}_{\alpha,a,b} := \lim_{N\to \infty} \meanval{ \frac{1}{N}\mathit{Tr}(J_\alpha^l) }_{J_\alpha} = r_l(\alpha ) \, .
		\end{align}

	\end{corollary}

	\begin{proof}
		We will just prove \eqref{eq:final_mom_alpha}-\eqref{eq:final_mom_beta} since the proof of the other cases is analogous.
		
		From Lemma \ref{lem:relation} and Theorem \ref{thm:structure_moments} one gets that:
		
		\begin{equation}
		v_\alpha^{(l)} =\lim_{N\to \infty}\left[ \frac{1}{N}\left(\sum_{j= \lfloor{l/2}\rfloor+1}^{N- \lfloor{l/2}\rfloor-1} w_l(\alpha) + O(1)\right)\right] = w_l(\alpha)\, .
		\end{equation}
		Indeed neglecting the terms $h_j^{(l)}$ $j=1,\, \ldots, \lfloor l/2\rfloor, N- \lfloor l/2\rfloor,\, \ldots, N$  in the average of $\mathit{Tr}\left(H_\alpha^l\right)$ we get an error of order $O(1)$ since $l$ is fixed, so in the summations we are neglecting a finite number of terms of order $O(1)$, see Remark \ref{rem:tutto_finito}.
		
		For the same reason one gets that:
		\begin{equation}
		u_\alpha^{(l)} = \lim_{N\to \infty}\left[ \frac{1}{N}\sum_{j= \lfloor{l/2}\rfloor+1}^{N- \lfloor{l/2}\rfloor-1} w_l\left(\alpha\left(1-\frac{j}{N}\right)\right) + O(N^{-1}) \right]\, .
		\end{equation}
		
		Thus taking the limit for $N$ going to infinity one gets  { the integral in \eqref{eq:final_mom_alpha}. }
	\end{proof}

\begin{remark}
	\label{rem:mom}
			{We stress that $u_\alpha^{(l)}, u_{\alpha,\gamma}^{(l)} $ and $u_{\alpha,a,b}^{(l)}$ are respectively the $l^{th}$ moments of the Gaussian, Laguerre and Jacobi $\beta$-ensembles in the high temperature regime.
			Analogously, the quantities $v_\alpha^{(l)}, v_{\alpha,\gamma}^{(l)} $ and $v_{\alpha,a,b}^{(l)}$ are the $l^{th}$ moments of the Gaussian, Laguerre and Jacobi $\alpha$-ensembles respectively.}
\end{remark}
	
	We can now finish the proof of Thereon \ref{THM:MAIN}
	\begin{proof}[Proof of Theorem \ref{THM:MAIN}]
		From Corollary \ref{cor:momenti_alpha} one {concludes}  that for all fixed $l\in \N$:
		
		\begin{align}
		\label{eq:divergence_alpha}
		& v^{(l)}_{\alpha} = \partial_{\alpha}(\alpha u^{(l)}_\alpha)\,,\\
		\label{eq:divergence_alpha1}
		& v^{(l)}_{\alpha,\gamma} = \partial_\alpha (\alpha  u^{(l)}_{\alpha,\gamma})\, ,\\
		\label{eq:divergence_alpha2}
		& v^{(l)}_{\alpha,a,b} = \partial_\alpha (\alpha u^{(l)}_{\alpha,a,b})\, .
		\end{align}

		 By Theorem~\ref{thm:ht_regime}, Corollary \ref{cor:momenti_alpha} and Remark \ref{rem:mom}, the quantities $u^{(l)}_\alpha$, $u^{(l)}_{\alpha,\gamma}$ and $u^{(l)}_{\alpha,a,b}$ are the moments of the measures with density  $\mu_\alpha$, $\mu_{\alpha,\gamma}$ and $\mu_{\alpha,a,b}$ defined in 
			\eqref{eq:mean_gaussian}, \eqref{eq:mean_laguerre} and \eqref{eq:mean_jacobi}.
			Moreover by formula  \eqref{eq:divergence_moments}  such moments uniquely determine the corresponding measures.
			
			It follows from relation \eqref{eq:divergence_alpha} and Lemma \ref{lem:bay} that    the mean density of states 
			$\wo{\di \nu_H}$ of the Gaussian $\alpha$-ensemble
			coincides with  $\partial_{\alpha}(\alpha \mu_\alpha)$ with $\mu_\alpha$ as   \eqref{eq:mean_gaussian}.
			In a similar way, by \eqref{eq:divergence_alpha1},   the measure 
			$\partial_{\alpha}(\alpha u_{\alpha,\gamma})$ in \eqref{eq:mean_laguerre} is the mean density of states $ \wo{\di \nu_L}$ of the Laguerre $\alpha$-ensembles and $\partial_{\alpha}(\alpha u_{\alpha,a,b})$ 
			in \eqref{eq:mean_jacobi} is the mean density of states  $\wo{\di \nu_H}$ of
			the Jacobi $\alpha$-ensembles.
			
	Since for the $\alpha$-ensembles all $\{a_n\}_{n=1}^N$ and $\{b_n\}_{n=1}^{N-1}$ have all finite moments, one can apply Theorem \ref{thm:nakano} getting that the moments of the empirical spectral distributions of the $\alpha$- ensembles  $\di\nu_H^{(N)},\di\nu_L^{(N)}$ and $\di\nu_J^{(N)}$ converge almost surely to the ones of the corresponding mean density of states  	$\wo{\di \nu_H}, \wo{\di \nu_L}$ and $  \wo{\di \nu_J}$ in \eqref{eq:mean_gaussian}, \eqref{eq:mean_laguerre}  and \eqref{eq:mean_jacobi} respectively. 
	Furthermore applying Lemma \ref{lem:duy} one obtains that   the spectral distributions of the $\alpha$-ensembles  $\di\nu_H^{(N)},\di\nu_L^{(N)}$ and $\di\nu_J^{(N)}$ converge almost surely to 	$\wo{\di \nu_H}, \wo{\di \nu_L}$ and $ \wo{\di \nu_J}$ in \eqref{eq:mean_gaussian}-\eqref{eq:mean_laguerre} and \eqref{eq:mean_jacobi} respectively.	
	
	Finally from \eqref{eq:fluctuations} one gets that formula \eqref{eq:fluct_gaussian}--\eqref{eq:fluct_jacobi} hold, namely that the global fluctuations are Gaussian.
	\end{proof}

	\subsection{Parameters limit}
	
	In this section we study the behavior of  $\alpha$-ensembles when the parameter $\alpha$ goes to infinity.  For this purpose  we  consider  the {rescaled} version of {$\alpha$-}ensembles, i.e. the matrices defined as
	$\frac{1}{\sqrt{\alpha}}H_\alpha, \frac{\gamma}{\alpha}L_{\alpha,\gamma}$ and $J_\alpha$.  The corresponding mean density of states is rescaled to  $\di \wo \nu_H(\sqrt{\alpha }x), \di \wo \nu_L\left(\frac{\alpha x}{\gamma}\right)$ and
	 $\di \wo \nu_J(x)$ (see \eqref{eq:mean_gaussian}--\eqref{eq:mean_jacobi}).
	
	Now we have to compute the limits of these measures when $\alpha \to \infty$. We will compute these limits using the matrix representations of the normalized $\alpha$-ensemble and exploit the following weak limits:
	
	\begin{equation}
	\label{eq:weak_lim_distr}
	\lim_{\alpha\to \infty} \frac{\cN(0,2)}{\sqrt{\alpha}} \stackrel{d}{\to} 0 \qquad \lim_{\alpha\to \infty}\frac{\chi_\alpha}{\sqrt{\alpha}} \stackrel{d}{\to} 1 \qquad \lim_{\alpha\to \infty} \text{Beta}(\alpha,\alpha) \stackrel{d}{\to} \frac{1}{2}.
	\end{equation} 
	
	The above relations imply that the mean density of states of the three normalized $\alpha$-ensembles weakly converges to the mean density of states of the following matrices:
	
	\begin{equation}
	\begin{split}
	H_\infty &= \begin{pmatrix}
	0 & 1 \\
	1 & 0& 1 \\
	& \ddots & \ddots & \ddots\\
	&& \ddots & \ddots & 1 \\
	&&& 1 & 0
	\end{pmatrix}\,,\qquad 	L_\infty = \begin{pmatrix}
	1 &   \sqrt{\gamma} \\
	\sqrt{\gamma} & 1+\gamma & \sqrt{\gamma} \\
	& \ddots & \ddots & \ddots\\
	&& \ddots & \ddots & \sqrt{\gamma} \\
	&&& \sqrt{\gamma} & 1+\gamma
	\end{pmatrix}\, , \qquad \\	J_\infty &= \begin{pmatrix}
	\frac{1}{\sqrt{2}} & \frac{1}{2\sqrt{2}} \\
	\frac{1}{2\sqrt{2}} & \frac{1}{2}& \frac{1}{4} \\
	& \frac{1}{4}& \ddots & \ddots\\
	&& \ddots & \ddots & \frac{1}{4} \\
	&&& \frac{1}{4} & \frac{1}{2}
	\end{pmatrix}\, .
	\end{split}
	\end{equation}
	
		The eigenvalues distributions of the above matrices in the large $N$ limit are given by 
	
	\begin{align}
	\label{eq:lim_gaussian}
	&\lim_{\alpha\to \infty} \di \wo \nu_H (\sqrt{\alpha}x)= \frac{\mathbbm{1}_{(-2,2)}}{\pi \sqrt{4-x^2}}\di x\, , \\
	& \lim_{\alpha \to \infty}\di \wo \nu_L\left(\frac{\alpha x}{\gamma}\right) = \frac{\mathbbm{1}_{((1-\sqrt{\gamma})^2,(1+\sqrt{\gamma})^2)}}{\pi\sqrt{4\gamma-(x-1-\gamma)^2}} \di x \, , \\\\
	\label{eq:lim_jacobi}
	& \lim_{\alpha \to \infty} \di \wo \nu_J(x) = \frac{2\mathbbm{1}_{(0,1)}}{\pi \sqrt{1-(2x-1)^2}} \di x\, ,
	\end{align}
	where $\mathbbm{1}_{(a,b)}$ is the indicator function of the interval $(a,b)$.
	
	We observe  that for all the three  $\alpha$-ensembles in the large $\alpha$ limit,  the  corresponding mean density of states  is an arcsine distribution.
		It would be interesting to study the behavior of the fluctuations of the  max/min eigenvalue of the  $\alpha$-ensembles  in the limit  of large $\alpha$.

	\section{An application to the Toda chain}
	
	In this section we will apply  Theorem \ref{THM:MAIN} to find the mean density of states of the classical Toda chain \cite{Toda1970} with periodic boundary conditions. As we already mentioned this is an alternative proof of the result in \cite{Spohn2019}.
	
	\subsection{Integrable Structure}
	
	The classical Toda chain is the  dynamical system described by the following Hamiltonian:
	\begin{equation}
	H_{T}(\bp,\bq):=\frac{1}{2} \sum_{j=1}^{N}{p_j^2} + \sum_{j=1}^{N}V_T(q_{j+1}-q_{j} )\ , \quad  V_T(x) =  e^{- x} +  x - 1\, ,
	\label{toda}
	\end{equation}
	
	with periodic boundary conditions  $ \, q_{j+N} = q_{j} \quad \forall \, j \in \mathbb{Z}$. Its equations of motion take the form
	\begin{equation}
	\label{Todaeq}
	\dot{q}_j=\dfrac{\partial H_T}{\partial p_j}=p_j,\quad \dot{p}_j=-\dfrac{\partial H_T}{\partial q_j}=V_T'(q_{j+1}-q_{j}) -V_T'(q_{j}-q_{j-1}), \;\;j=1,\dots,N\, .
	\end{equation}
	
	It is well known that the Toda chain is an integrable system \cite{Toda1970,Henon1974}, one way to prove it is to put the Toda equations  in Lax pair form. This was introduced by Flaschka \cite{Flaschka1974} and  Manakov \cite{Manakov1975} through the following {\em non canonical} change of coordinates:
	
	\begin{equation}
	a_j := -p_j \, , \qquad b_j:=  e^{\frac{1}{2}(q_j-q_{j+1})} \equiv e^{-\frac{1}{2}r_j}, \qquad 1 \leq j \leq N\,,
	\label{bavariable}
	\end{equation}
	where $r_j=q_{j+1}-q_j$ is the relative distance.  The periodic boundary conditions imply	
	\begin{equation}
		\sum_{j=1}^N r_j = 0\,.
	\end{equation}

	Then, defining the Lax operator $L$ as the periodic Jacobi matrix
	\cite{VanMoerbeke1976}
	\begin{equation} \label{eq:L_toda}
	L(\bb,\ba) := \left( \begin{array}{ccccc}
	a_{1} & b_{1} & 0 & \ldots &  b_{N} \\
	b_{1} & a_{2} & b_{2} & \ddots & \vdots \\
	0 & b_{2} & a_{3} & \ddots & 0 \\
	\vdots & \ddots & \ddots & \ddots & b_{N-1} \\
	b_{N} & \ldots & 0 & b_{N-1} & a_{N} \\
	\end{array} \right) \, ,
	\end{equation}
	
	and the anti-symmetric matrix $B$ 
	
	\begin{equation}
	B(\bb) := \left( \begin{array}{ccccc}
	0 & b_{1} & 0 & \ldots &  -b_{N} \\
	-b_{1} & 0 & b_{2} & \ddots & \vdots \\
	0 & -b_{2} & 0 & \ddots & 0 \\
	\vdots & \ddots & \ddots & \ddots & b_{N-1} \\
	b_{N} & \ldots & 0 & -b_{N-1} & 0 \\
	\end{array} \right) \, ,
	\end{equation}
	
	a straightforward calculation shows that the equations of motions \eqref{Todaeq} are equivalent to
	\begin{equation}
	\frac{d L}{dt} = \left[B;L\right]\, ,
	\end{equation}
	
	so the eigenvalues of $L$ are a set of integrals of motion.  
	
	\subsection{Gibbs ensemble and the density of states for the periodic Toda chain}
	
We consider the evolution of the Toda chain on the subspace:
	
	\begin{equation}
	\label{eq:subsapce}
		\cM := \left\{(\bp,\br) \in \R^N \times \R^N \, : \quad \sum_{j=1}^N r_j = \sum_{j=1}^N p_j = 0\right\}\, ,
	\end{equation}  
	which is invariant for the dynamics.  Indeed the condition  $ \sum_{j=1}^N r_j =0$ follows from the periodic boundary  conditions and the  condition  $\sum_{j=1}^N p_j = 0$ follows from the fact that the system 
	is translational invariant and therefore the total momentum is conserved.	
	We endow the phase space $\cM$ \eqref{eq:subsapce} with the Gibbs measure  for the Toda lattice at temperature $\beta^{-1}$  as
	
	\begin{equation}\label{eq:misura_toda}
	\di \nu_{Toda} :=    \frac{1}{Z_{Toda}(\beta)} \  e^{-\beta {H_T(\bp,\br)} }\delta_{\sum_{j=1}^N p_j }\delta_{\sum_{j=1}^N r_j } \di \bp \,  \di \br \, ,
	\end{equation}
	here $Z_{Toda}(\beta)$ is the partition function which normalize the measure. 
	
	We notice that this ensemble makes $L$ \eqref{eq:L_toda} into a random matrix, thus it makes sense to study its mean density of states. However the matrix entries of $L$  are not independent random variables because of  \eqref{eq:subsapce}. 
	For this reason we also introduce the {\em approximate} measure $\di \wt \nu_{Toda}$ on $\R^N\times \R^N$ as
	\begin{equation}
			\di \wt \nu_{Toda} :=    \frac{1}{\wt Z_{Toda}(\beta)} \  e^{-\beta {H_T(\bp,\br) - \theta \sum_j r_j} } \di \bp \,  \di \br \, ,
	\end{equation}
	where $ \wt Z_{Toda}(\beta)$ is the partition function which normalizes the measure and $\theta > 0$ is chosen in such a way that:
	\begin{equation}
		\la r_j, \di \wt \nu_{Toda} \ra = 0\, .
	\end{equation}
	The value of $\theta>0$ is unique for all $\beta >0$ since
	
	\begin{equation}
		\la r_j, \di \wt \nu_{Toda} \ra = \log(\beta) - \frac{\Gamma'(\beta + \theta)}{\Gamma(\beta + \theta)}\, ,
	\end{equation} 
	which has just one positive solution.
	
	From now on we will write $L$ and $\wt L$ as the random matrices whose entries are distributed according to the probability measure $\di \nu_{Toda}$ and $\di \wt \nu_{Toda}$ respectively.  
	 In particular applying the change of coordinates  \eqref{bavariable} one  gets that
	
	\begin{equation}
	\wt L \sim \frac{1}{\sqrt{2\beta}}\left( \begin{array}{ccccc}
	a_{1} & b_{1} & 0 & \ldots &  b_{N} \\
	b_{1} & a_{2} & b_{2} & \ddots & \vdots \\
	0 & b_{2} & a_{3} & \ddots & 0 \\
	\vdots & \ddots & \ddots & \ddots & b_{N-1} \\
	b_{N} & \ldots & 0 & b_{N-1} & a_{N} \\
	\end{array} \right), \quad b_j \sim \chi_{2(\beta + \theta)}\, , a_j \sim \cN(0,2)\quad j=1,\ldots,N.
	\end{equation}

	To obtain the mean density of states of the Toda lattice with periodic boundary conditions we need the following lemma, whose proof can be found in \cite{Grava2020}:
	\begin{lemma}(cf. \cite[Lemma 4.1]{Grava2020})
	\label{magic}
	Fix $\wt\beta >0$  and let $f \colon \R^N\times \R^N \to \R$ depends on just $K$ variables and  finite second order moment with respect to $\di \wt\nu_{Toda}$, uniformly for all $\beta > \wt\beta$.
	Then there exist   positive constants $C, N_0$ and $\beta_0 $ such that
	for all $N > N_0$, $\beta > max\{\beta_0,\tilde{\beta}\}$ one has  
	\begin{equation}
	\label{eq:magic}
	\abs{\la f, \di\nu_{Toda} \ra - \la f,\di\wt\nu_{Toda} \ra }\leq  C \frac{K}{N} \sqrt{\la f^2,\di\wt\nu_{Toda}\ra - \la f,\di\wt\nu_{Toda} \ra^2}\, .
	\end{equation}
	\end{lemma}
	
	Applying this Lemma we can conclude that the matrices  $L$ and $\wt L$  have the same moment sequence in the large $N$ limit.
	Furthermore, $\wt L$ is a rank one perturbation of the matrix $\frac{1}{\sqrt{\beta}}H_{\theta+\beta}$ in table \ref{tab:alpha_ensemble}.
	  So we can use the following theorem, whose proof can be found in  \cite{Bai2010}, to { show}  that the mean density of states of the  matrices $\wt L$ and  $\frac{1}{\sqrt{\beta}}H_{\theta+\beta}$  in the large $N$ limit are the same.
	\begin{theorem}(cf. \cite[Theorem A.43]{Bai2010})
		Let $A,B$ be two $N\times N$ Hermitian matrices and $F^A,F^B$ their empirical spectral density defined as:
		
		\begin{equation}
		F^A(x) := \frac{1}{N}\#\{j\leq N \, : \, \lambda_j \leq x \}\,,
		\end{equation}
		where $\lambda_j$ are the eigenvalues of $A$. Then
		
		\begin{equation}
		|| F^A -F^B|| \leq \frac{1}{N}\text{ Rank}(A-B)\, ,
		\end{equation} 
		where $||f|| = \sup_x |f(x)|$.
	\end{theorem}

	This implies also that the moment sequence of $\wt L$ and $  \frac{1}{\sqrt{\beta}}H_{\theta+\beta}$ are the same in the large $N$ limit, which means that also the moment sequence of $L, \frac{1}{\sqrt{\beta}}H_{\theta+\beta}$ in the large $N$ limit are equal. 
	So applying Lemma \ref{lem:bay} and Theorem \ref{THM:MAIN} one gets that :
	
	\begin{lemma}
		Consider the classical Toda chain \eqref{toda} and endow the phase space $\cM$ \eqref{eq:subsapce} with the Gibbs measure $\di \nu_{Toda}$ in \eqref{eq:misura_toda}, then there exists a  constant $\beta_0 > 0$ such that, for all $\beta > \beta_0$ the mean density of states of the Lax matrix $L$ \eqref{eq:L_toda} {in the limit $N\to\infty$} is explicitly given by:

		\begin{equation}
		\label{eq:mean_toda}
		\wo{\di \xi_L}(x) =\sqrt{\beta} \partial_{\alpha}(\alpha\mu_{\alpha}(\sqrt{\beta}x))_{\rvert_{\alpha = \beta + \theta}} \di x \, ,
		\end{equation}
		\noindent
		where $\mu_\alpha(x)$ is given in \eqref{eq:spectralGaussian}.
	\end{lemma}
	To conclude, we  also remark that if we let the inverse temperature $\beta$   approach infinity, in view of \eqref{eq:lim_gaussian},  we obtain  that the mean density of states of the classical Toda chain in this regime is exactly the arcsine law \eqref{eq:mean_gaussian}.
	 {From the physical point of view the system is at rest.}    
	
	{ \bf Acknowledgments.}
	This project has  received funding from the European Union's H2020 research and innovation program under the Marie Sk\l owdoska--Curie grant No. 778010 {\em  IPaDEGAN}  and from LIA, LYSM, AMU, CNRS, ECM, INdAM. \\

	\begin{figure}
		
		\includegraphics[scale=0.5]{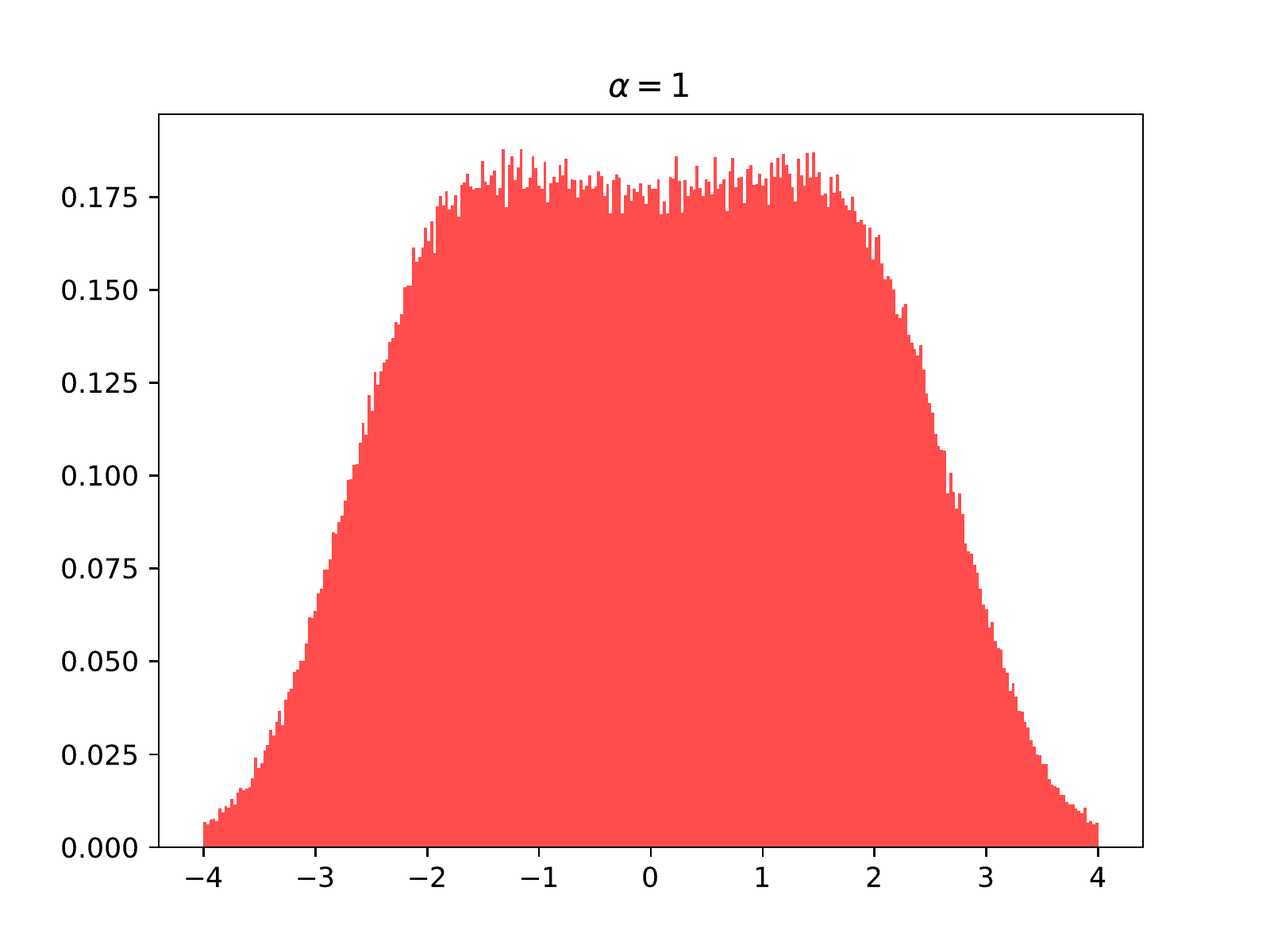}
		\includegraphics[scale=0.5]{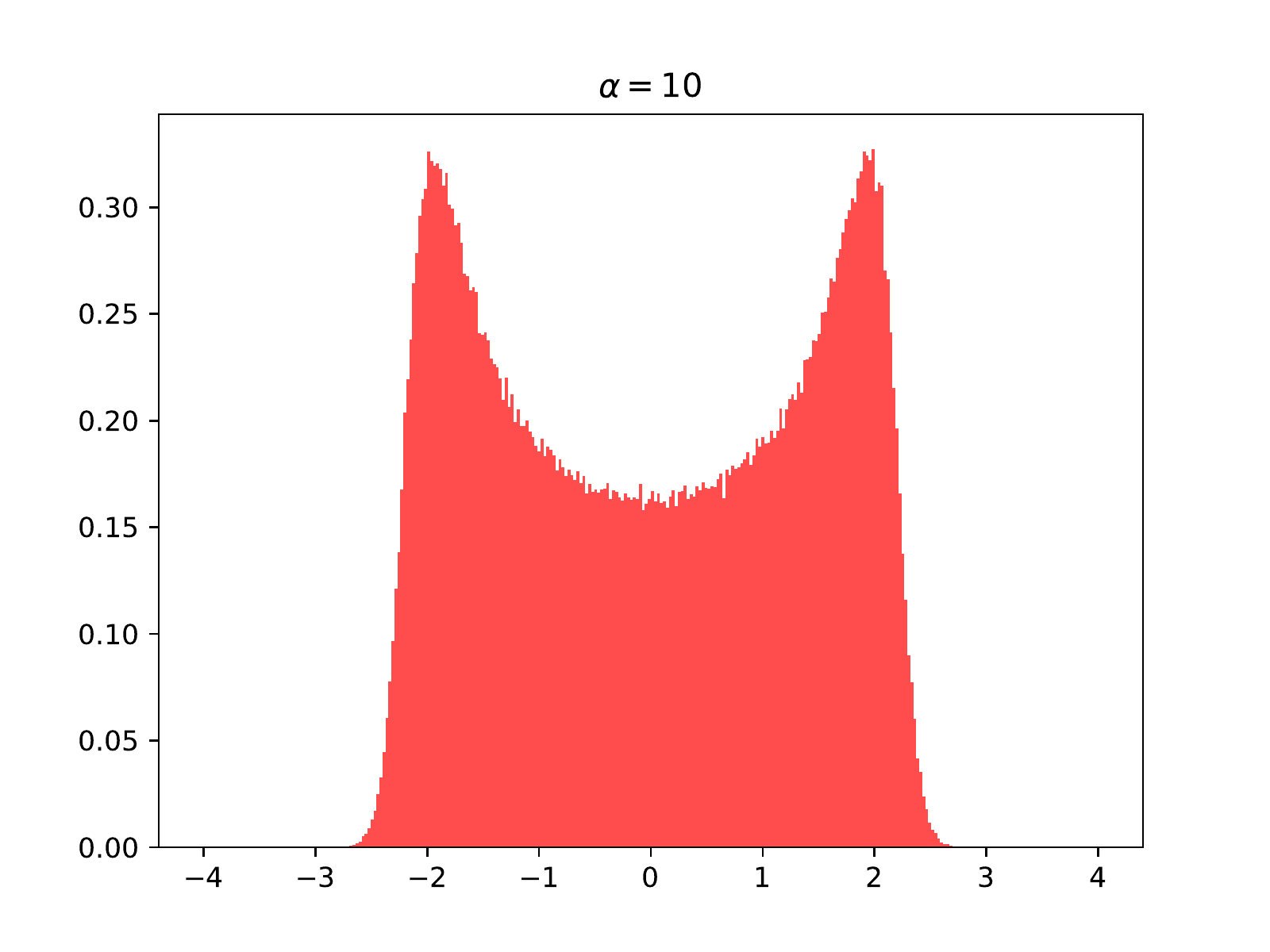}
		\includegraphics[scale=0.5]{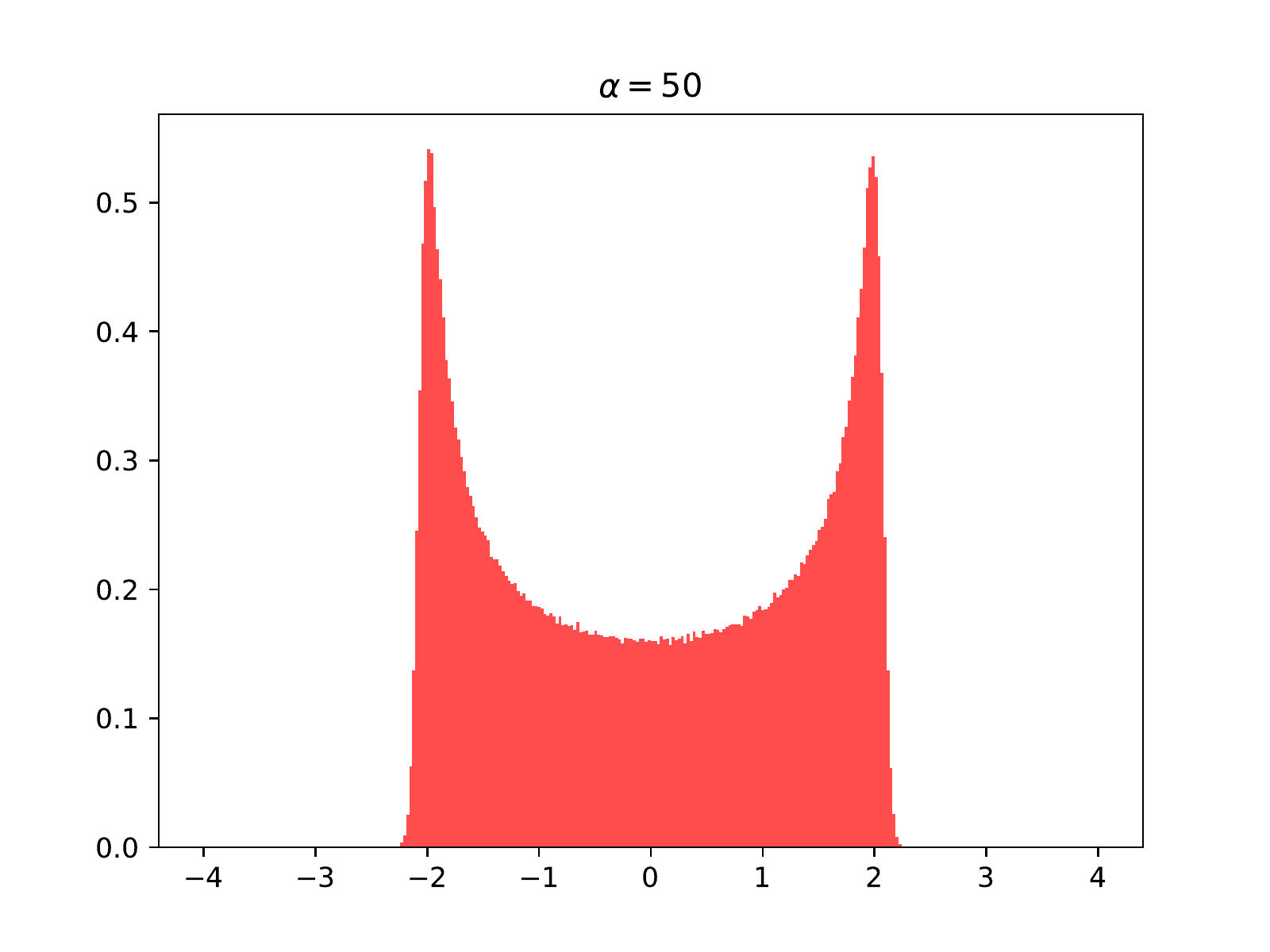}
		\includegraphics[scale=0.5]{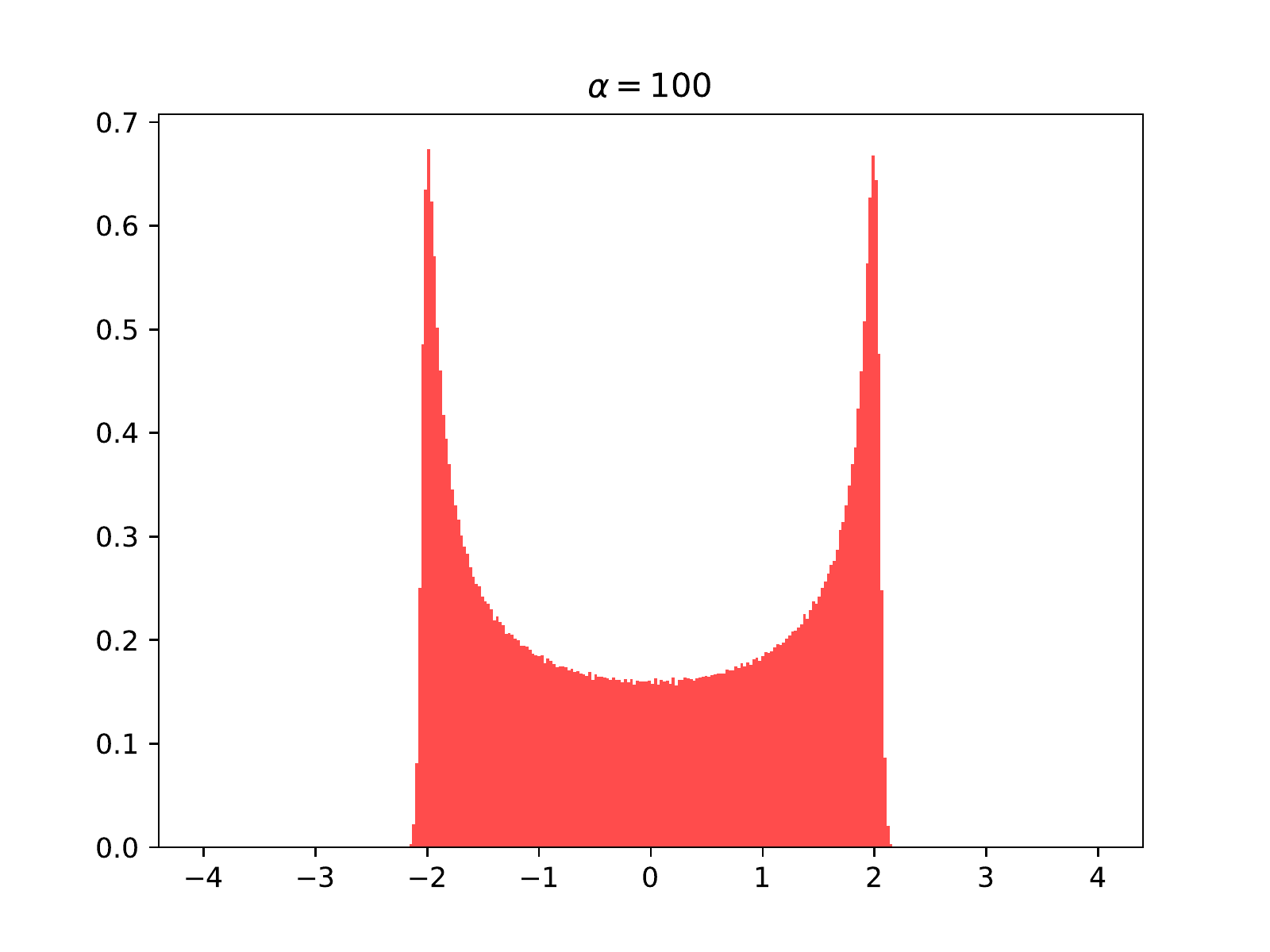}
		\caption{Gaussian $\alpha$ ensemble empirical spectral density for different values of the parameters, $N= 500$, trials: $5000$.}
		\label{fig:gaussian}
	\end{figure}
	\begin{figure}
		
		\includegraphics[scale=0.5]{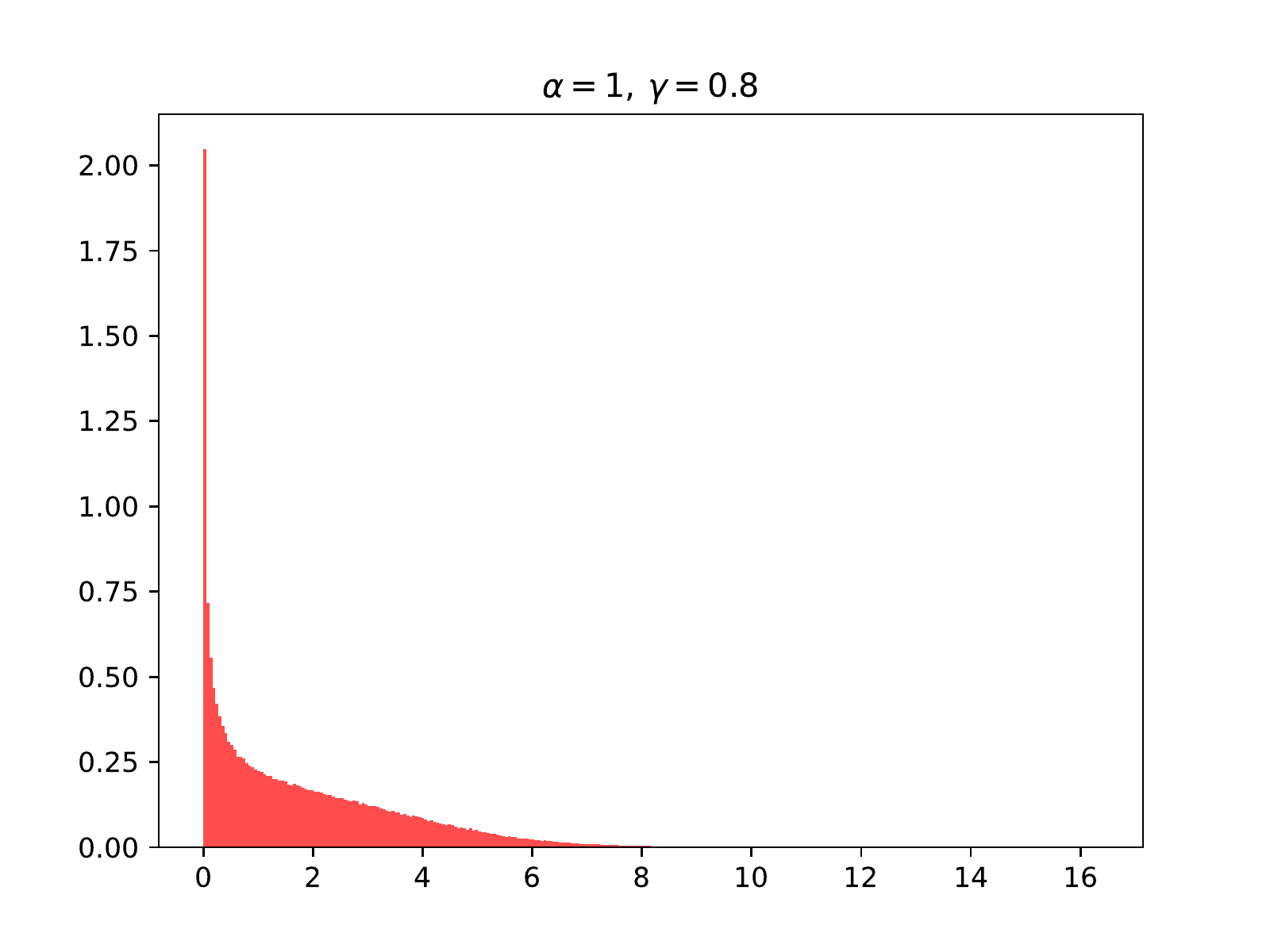}
		\includegraphics[scale=0.5]{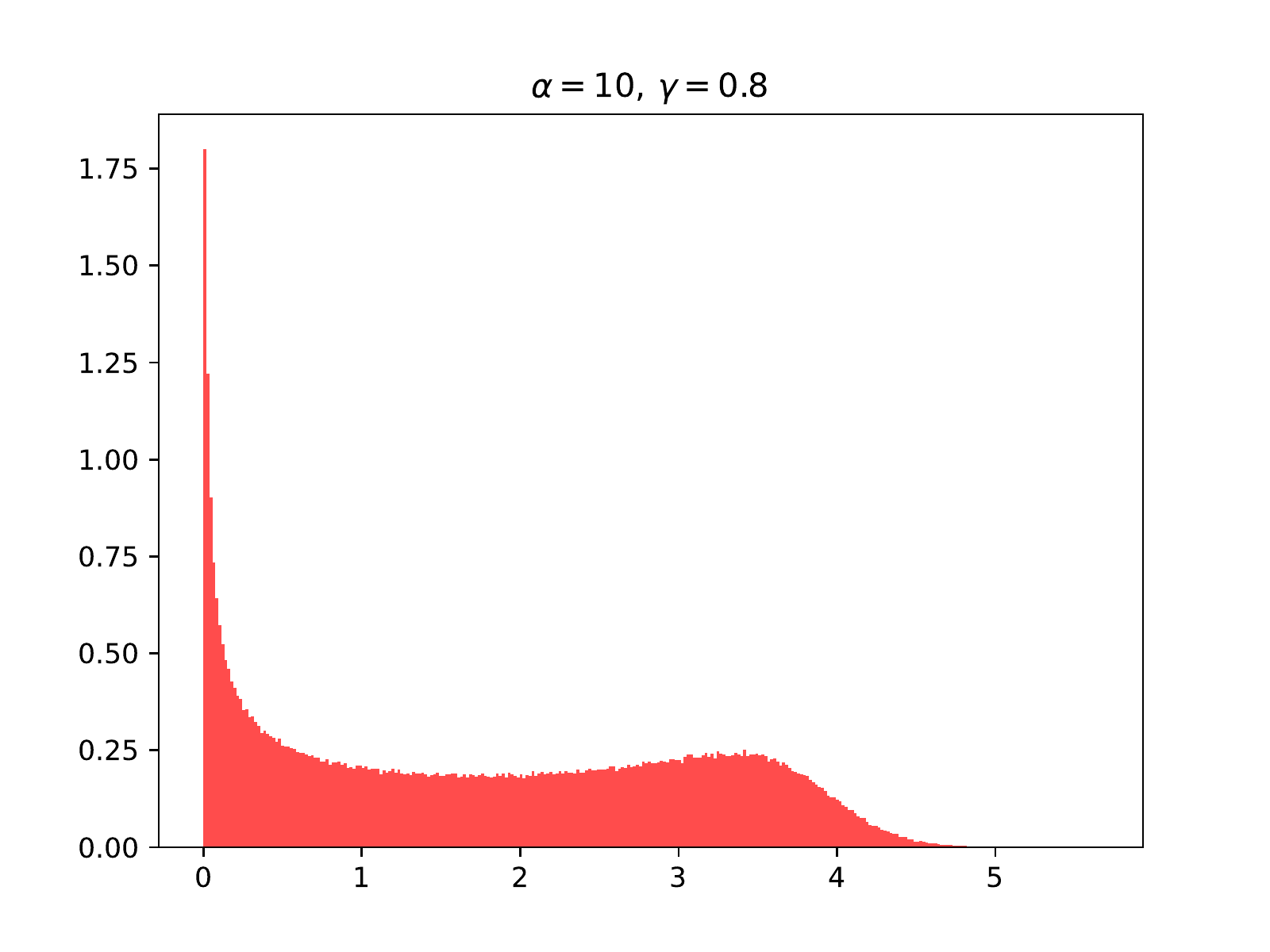}
		\includegraphics[scale=0.5]{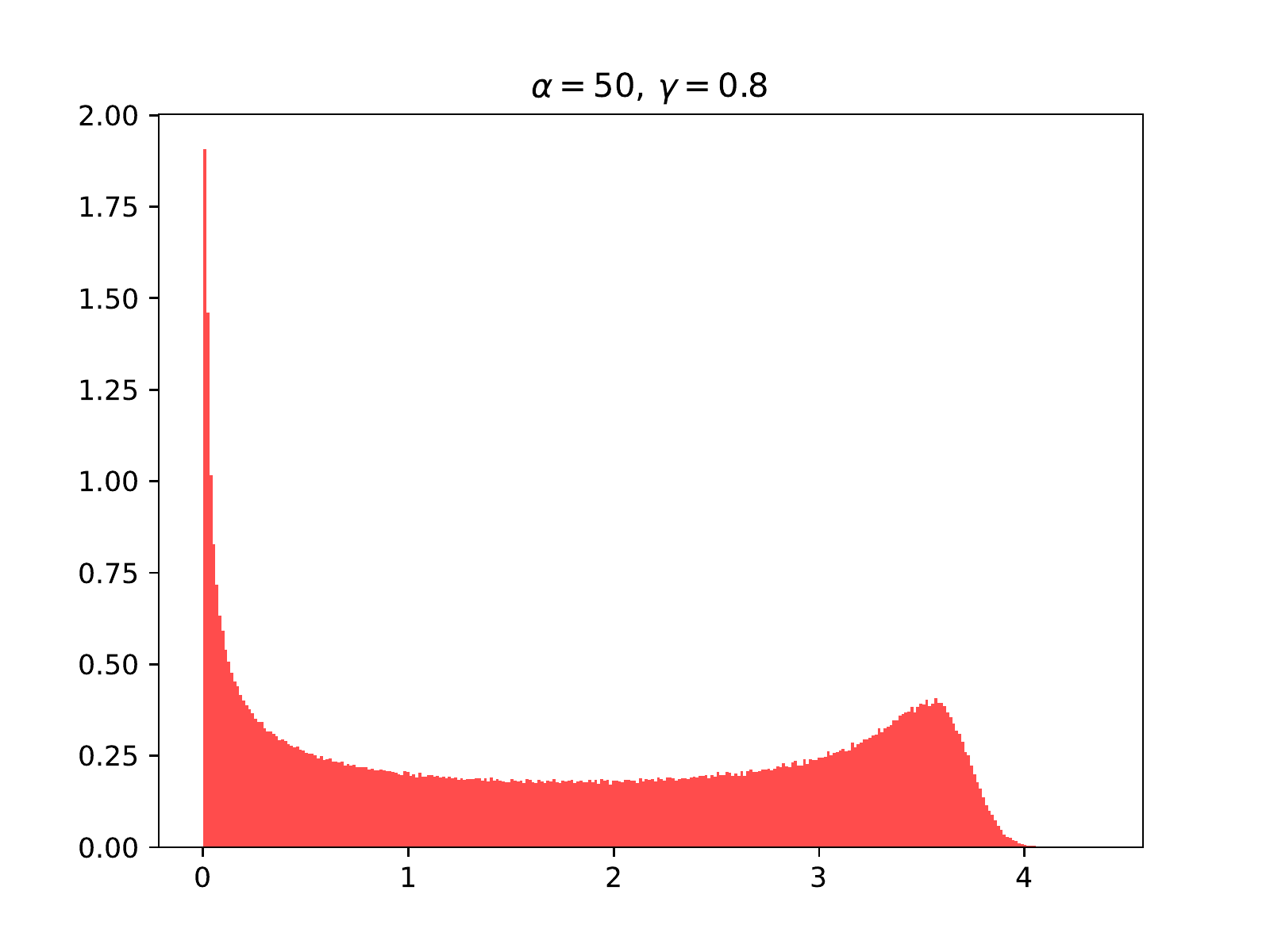}
		\includegraphics[scale=0.5]{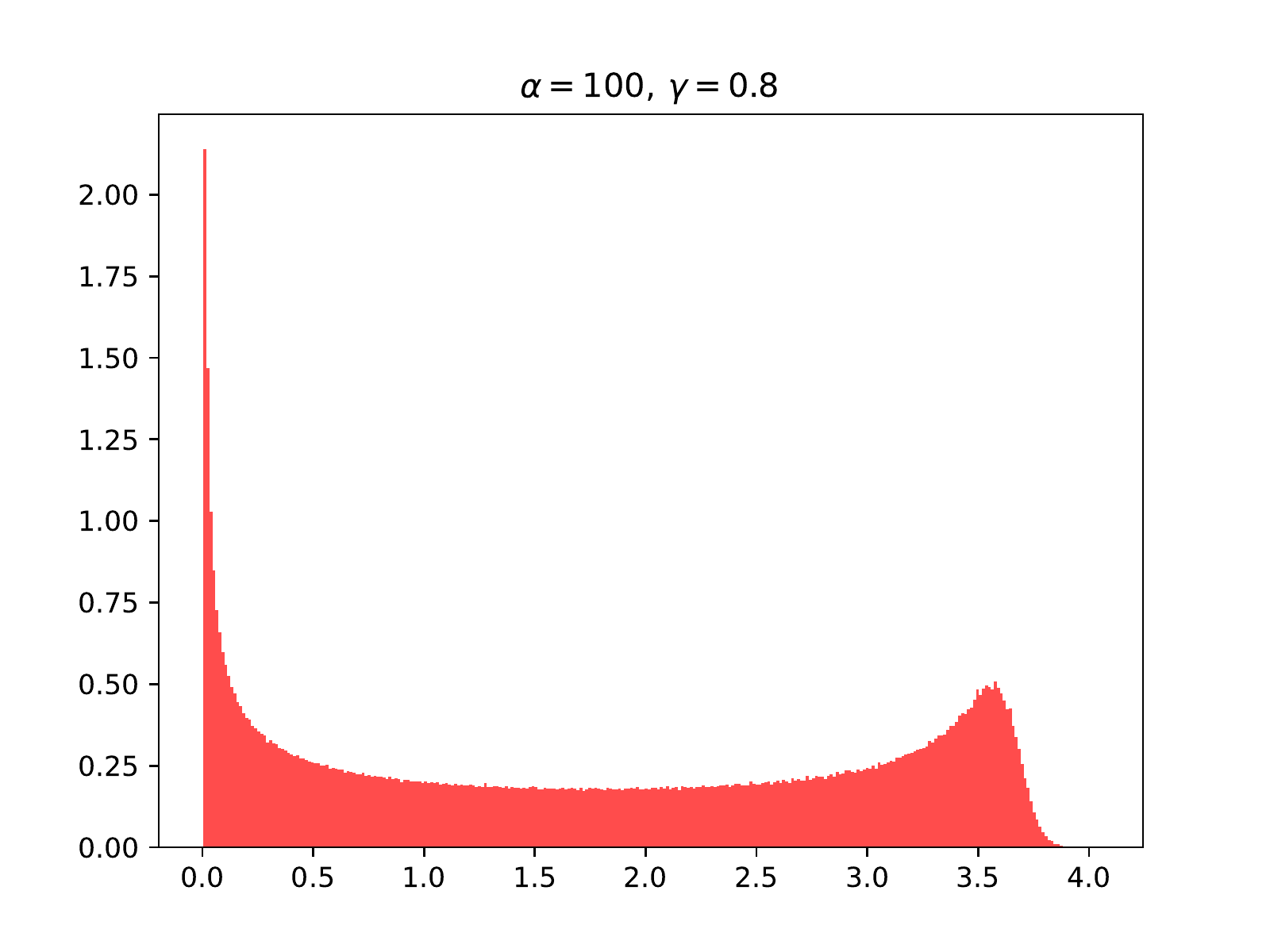}
		\caption{Laguerre $\alpha$ ensemble empirical spectral density for different values of the parameters, $N= 500$, trials: $5000$.}
		\label{fig:laguerre}
	\end{figure}
	\begin{figure}
		\includegraphics[scale=0.5]{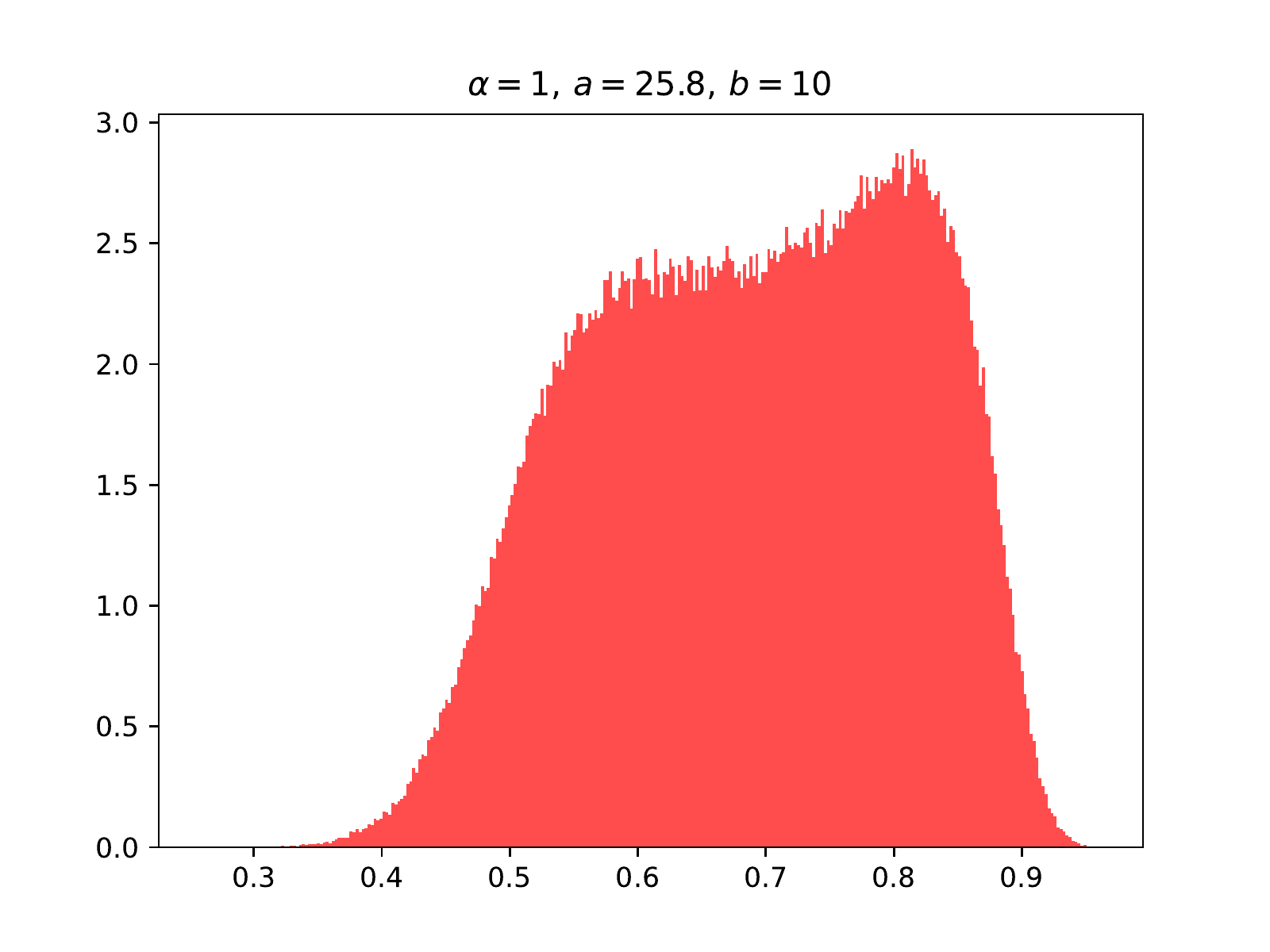}
		\includegraphics[scale=0.5]{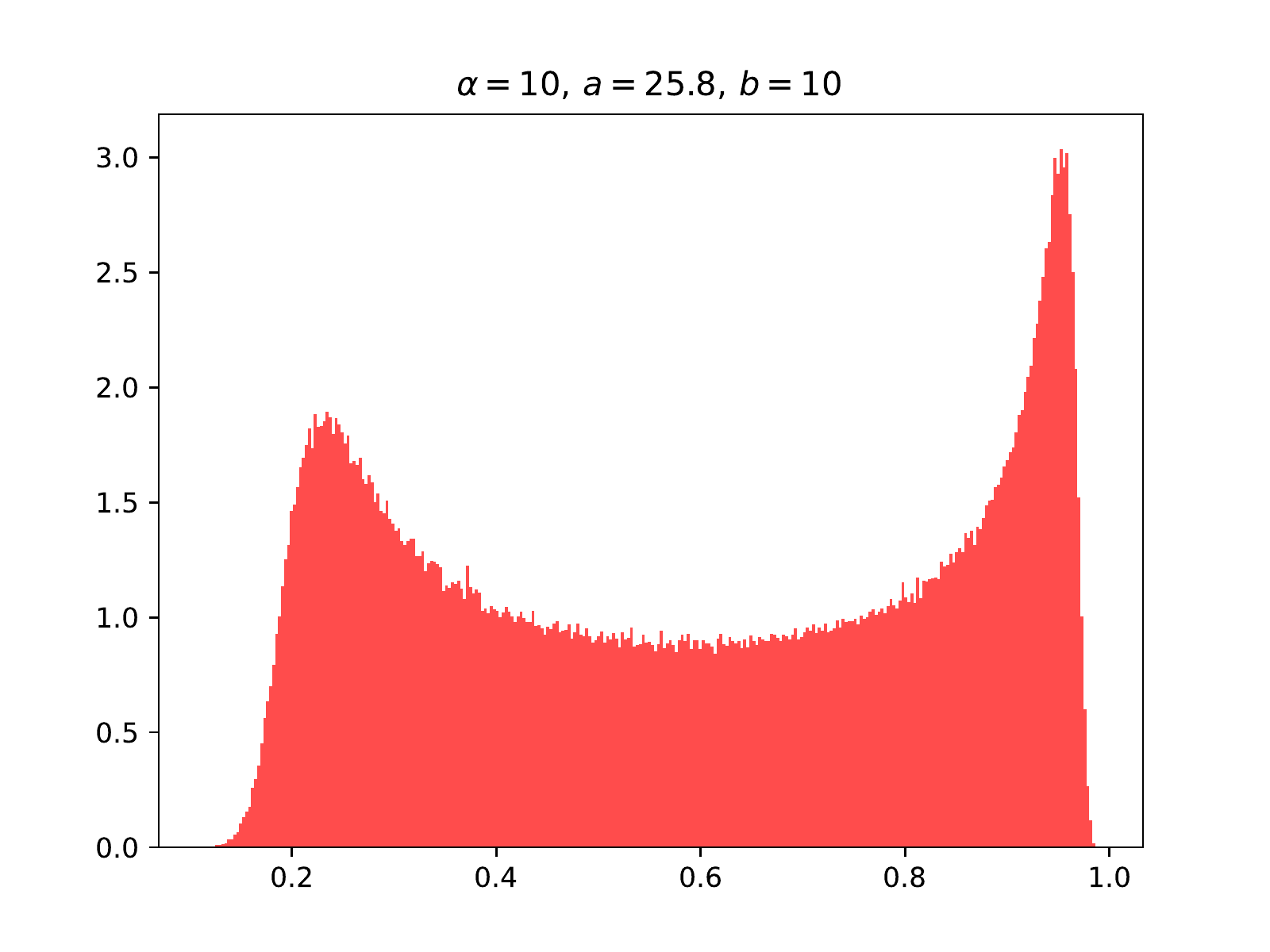}
		\includegraphics[scale=0.5]{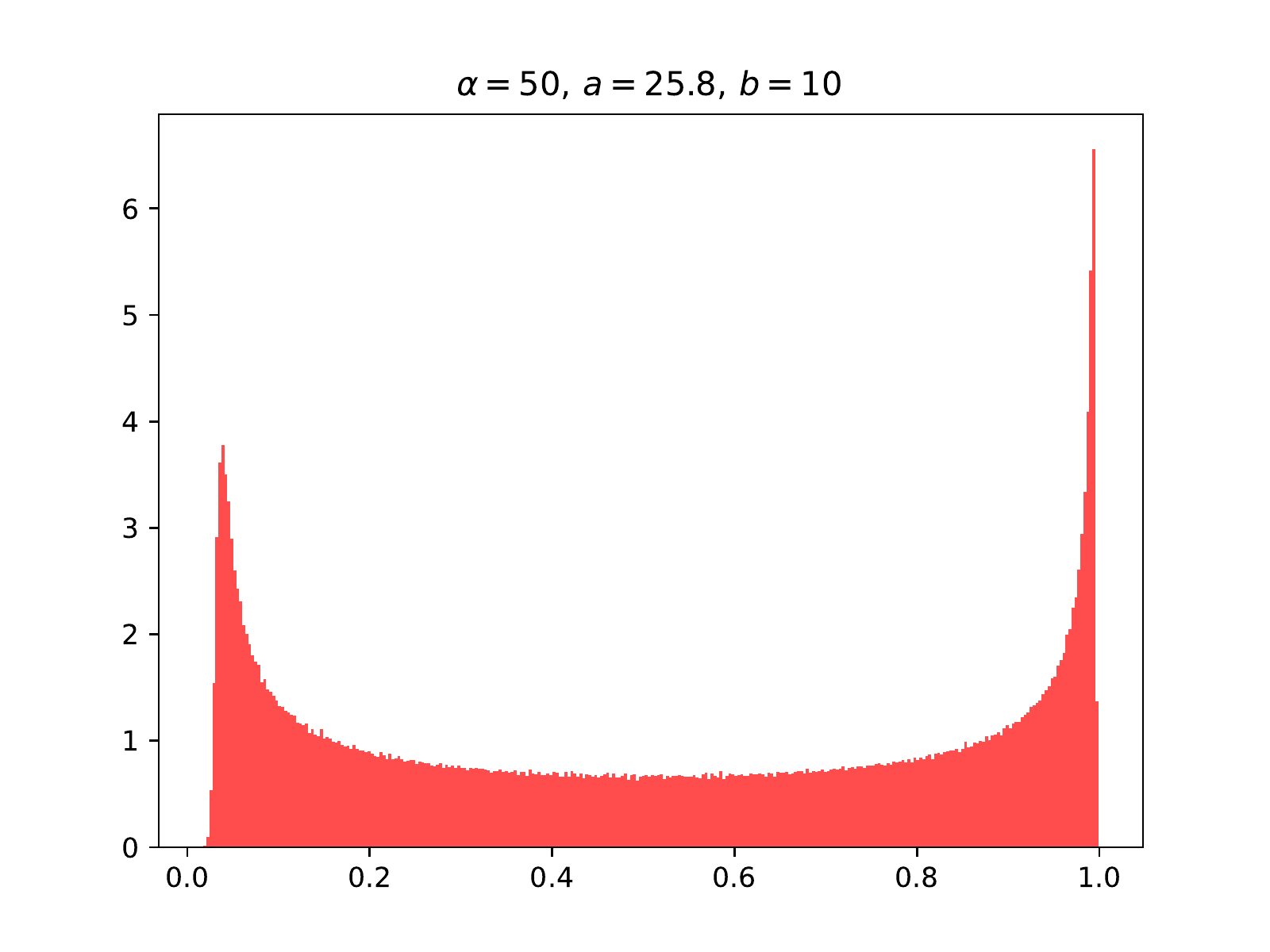}
		\includegraphics[scale=0.5]{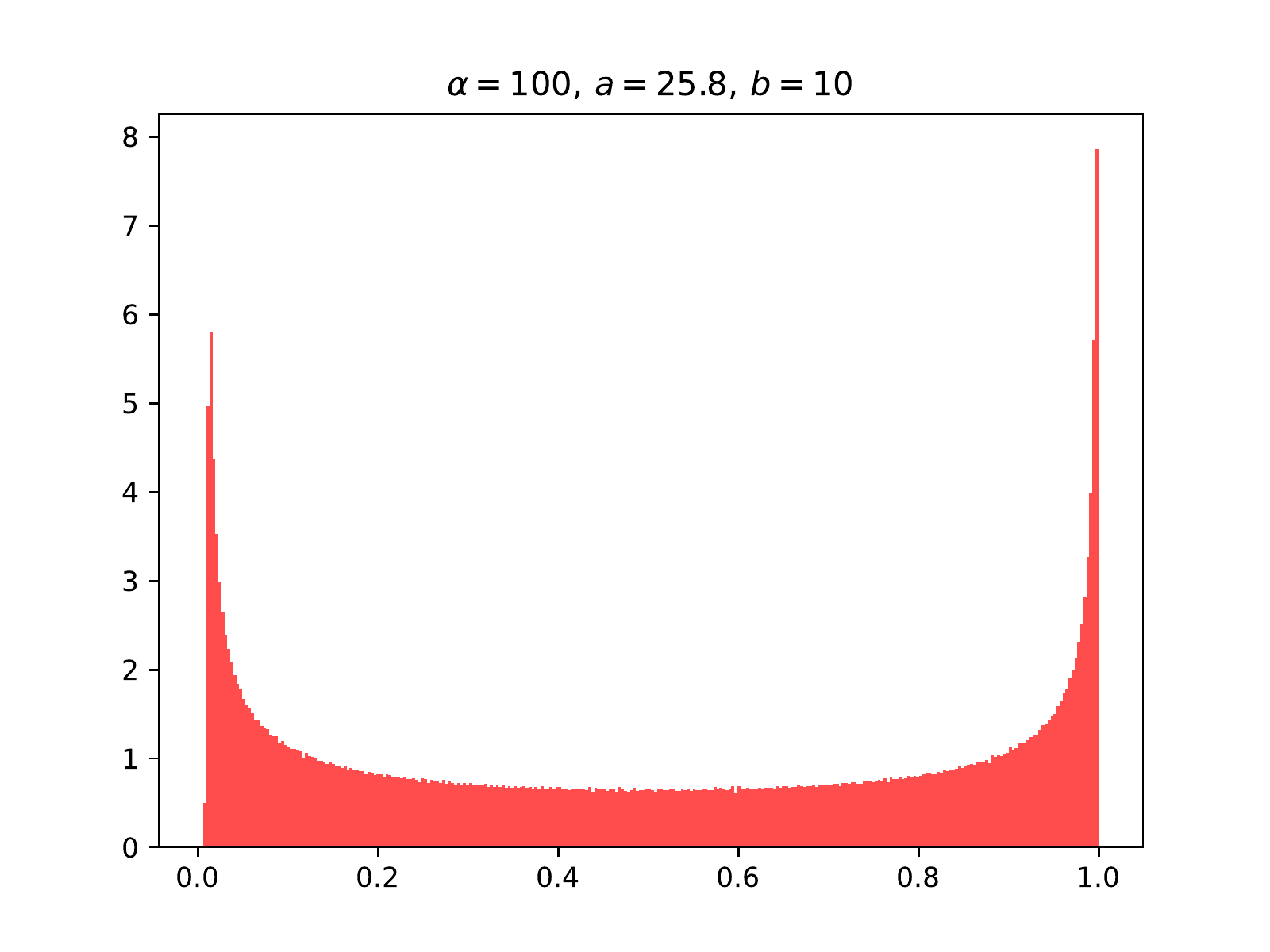}
		\caption{Jacobi $\alpha$ ensemble empirical spectral density  for different values of the parameters, $N= 500$, trials: $5000$.}
		\label{fig:jacobi}
	\end{figure}
	
	\newpage
	\appendix
	\section{ Associate Orthogonal polynomials} 
	\label{appendixA}

		% In \cite{Kerov} it was proved that this is also the orthogonality measure of the so called 
		The associate Hermite polynomials  $H^{(\alpha)}_{n}(x)$ were    introduced in \cite{Askey1984}. They are orthonormal polynomials   with respect to the measure $\mu_\alpha$ defined in \eqref{eq:spectralGaussian}, namely
		\[
		\int_{-\infty}^{+\infty}H^{(\alpha)}_{n}(x)H^{(\alpha)}_{n}(x)\mu_{\alpha}(x)dx=\delta_{nm},
		\]
		and satisfy the following   three terms recurrence relation:
		\begin{equation}
		xH^{(\alpha)}_{n}(x)=H^{(\alpha)}_{n+1}(x) +(n+\alpha)H^{(\alpha)}_{n-1}(x), \quad H_{-1}^{(\alpha)}(x) =0, \, \;H_0(x) =1 \, ,
		\end{equation} 
		for $\alpha=0$ one gets the standard Hermite polynomials.
		
		The  associate Laguerre polynomials of type 2, $L^{\alpha,\gamma}_n(x)$,  were introduced  \cite{Ismail1988}. They satisfy the orthogonality relation
		\[
		\int_{0}^{\infty}L^{\alpha,\gamma}_n(x)L^{\alpha,\gamma}_m(x)\mu_{\alpha,\gamma}(x)dx=\delta_{nm}\, ,
		\]
		where $$\mu_{\alpha,\gamma}(x)= \frac{1}{\Gamma(\alpha+1)\Gamma\left(1+ \frac{\alpha}{\gamma} + \alpha\right)} \frac{x^\frac{\alpha}{\gamma}e^{-x}}{\Big\lvert \psi\left(\alpha,-\frac{\alpha}{\gamma};xe^{-i\pi}\right)\Big\rvert^2}\, ,
		$$ 
		is  defined in \eqref{eq:spectralLaguerre}.  They also satisfy the   following three terms  recurrence relation:
		
		\begin{equation}
			L_0(x) = 1\, , \qquad L_1(x) = \frac{\alpha + \frac{\alpha}{\gamma} +1 - x}{\alpha +1}\, , 
		\end{equation}

		\begin{equation}
		- xL^{\alpha,\gamma}_n(x) =  (n+1+\alpha)L^{\alpha,\gamma}_{n+1} - \left(2n+\frac{\alpha}{\gamma} + \alpha +1\right)L^{\alpha,\gamma}_n(x)  + \left(n+\alpha + \frac{\alpha}{\gamma}\right)L^{\alpha,\gamma}_{n-1}(x) \, .
		\end{equation}
	
		In the definition of $\mu_{\alpha,\gamma}$ the Tricomi confluent hypergeometric function  $\psi(a,b;z)$ \cite{Abramowitz1972} is defined to be the standard solution of the Kummer's equation
		\begin{equation}
		z\frac{\di^2 \psi}{\di z^2} + (b-z)\frac{\di \psi}{\di z} - a \psi = 0\, ,
		\end{equation}
		uniquely determined by the normalization $\psi(a,b;z) \sim z^{-a}$ as $z\to \infty$ and $|\arg(z)| \leq \frac{3}{2}\pi$, here $\arg(z)$ is the argument of the complex number $z$.
		Moreover if $b\not\in \N$ then there exists an alternative formula for the Tricomi confluent hypergeometric:
		
		\begin{equation}
		\psi(a,b;z) = \frac{\Gamma(1-b)}{\Gamma(a-b+1)} {_1}F_1(a,b;z) + \frac{\Gamma(b-1)}{\Gamma(a)} {_1}F_1(a-b+1,2-b;z)\, ,
		\end{equation}
		where 
		
		\begin{equation}
		_1F_1(a,b;z) = \sum_{n=0}^\infty\frac{(a)_n}{(b)_nn!}z^n\,, \quad (a)_n = a(a+1)\cdots(a+n-1)\, .
		\end{equation}

		It was shown in \cite{Trinh2020}  that $\mu_{\alpha,a,b}$ \eqref{eq:explicit_jacobi_TT} is the orthogonality measure of the associate Jacobi polynomials of type 3,   $J^{{\alpha,a,b}}_n(x)$:
		\[
		\int_0^1J^{{\alpha,a,b}}_n(x)J^{{\alpha,a,b}}_m(x)\mu_{\alpha,a,b}(x)dx=\delta_{nm}.
		\]
		The polynomials $J^{{\alpha,a,b}}_n(x)$ satisfy the following recurrence relation:	
		\begin{equation}
		xJ^{{\alpha,a,b}}_n(x) = \sqrt{\xi_{n}\mu_{n+1}}J^{{\alpha,a,b}}_{n+1}(x)+(\xi_n+\eta_n)J^{{\alpha,a,b}}_n(x) + \sqrt{\xi_{n-1}\mu_n}J^{{\alpha,a,b}}_{n-1}(x)\, ,
		\end{equation}
		where
		\begin{equation}
		\begin{cases}
		\xi_0(\alpha) = \frac{\alpha + a +1}{2\alpha + a + b +2} \\
		\xi_n(\alpha)  = \frac{n + \alpha + a + 1}{2n+ 2 \alpha + a + b + 2} \frac{n + \alpha + a+ b +1  }{2n + 2\alpha + a + b + 1}, \, \quad n > 0 \\
		\eta_n(\alpha) = \frac{n + \alpha }{2n+ 2 \alpha + a + b + 1} \frac{n + \alpha + b   }{2n + 2\alpha + a + b }, \, \quad n > 0
		\end{cases}, \qquad \alpha \geq 0,\, a,b > -1\,,
		\end{equation}

	\newpage
	\bibliographystyle{siam}
	\bibliography{RMfix.bib}
\end{document}